\documentclass[12pt]{amsart}
\pdfoutput=1

\usepackage[utf8]{inputenc}
\usepackage[T1]{fontenc}
\usepackage{latexsym}
\usepackage{amsxtra}
\usepackage{amsmath}
\usepackage{amsfonts}
\usepackage{amssymb}
\usepackage{blkarray}
\usepackage{multirow}
\usepackage{pgf,tikz}
\usepackage{mathrsfs}
\usepackage{mathdots}
\usepackage{pdfpages}
\usepackage[colorlinks,linkcolor=blue,anchorcolor=black,citecolor=blue,filecolor=blue,menucolor=blue,runcolor=blue,urlcolor=blue]{hyperref}
\usepackage{geometry}
\definecolor{red}{rgb}{0,0,0}
\usetikzlibrary{arrows}

\tikzstyle{vertex}=[circle, draw, inner sep=0pt, minimum size=6pt]
\newcommand{\vertex}{\node[vertex]}

\newtheorem{theorem}{Theorem}[section]
\newtheorem{lemma}[theorem]{Lemma}
\newtheorem{corollary}[theorem]{Corollary}

\newtheorem{hypothesis}[theorem]{Hypothesis}

\numberwithin{equation}{section}

\theoremstyle{definition}
\newtheorem{definition}[theorem]{Definition}

\newtheorem{notation}[theorem]{Notation}

\theoremstyle{remark}

\newcommand{\bFp}{\mathbb{F}_p}

\newcommand{\AG}{\mathrm{AG}}
\newcommand{\PG}{\mathrm{PG}}
\newcommand{\GF}{\mathrm{GF}}
\newcommand{\F}{\mathbb{F}}

\newcommand{\AGE}{\mathrm{AG}(2,3)\backslash e}
\newcommand{\OVM}{\overline{\mathcal{M}(\Phi)}}

\DeclareMathOperator{\YT}{YT}

\makeatletter
\renewcommand*\env@matrix[1][r]{\hskip -\arraycolsep
  \let\@ifnextchar\new@ifnextchar
  \array{*\c@MaxMatrixCols #1}}
\makeatother

\begin{document}
\title[Highly connected ternary matroids]{On the Highly Connected Dyadic, Near-Regular, and Sixth-Root-of-Unity Matroids}

\author{Ben Clark}
\address{Department of Mathematics\\
Louisiana State University\\
Baton Rouge, Louisiana, USA}
\email{clarkbenj@myvuw.ac.nz}

\author{Kevin Grace}
\address{School of Mathematics, University of Bristol, Bristol, UK, and the Heilbronn Institute for Mathematical Research, Bristol, UK \newline Current Address: Department of Mathematics, Vanderbilt University,  Nashville, Tennessee, USA}
\email{kevin.m.grace@vanderbilt.edu}

\author{James Oxley}
\address{Department of Mathematics\\
Louisiana State University\\
Baton Rouge, Louisiana, USA}
\email{oxley@math.lsu.edu}

\author{Stefan H.M. van Zwam}
\address{Department of Mathematics\\
Louisiana State University\\
Baton Rouge, Louisiana, USA}
\email{stefanvanzwam@gmail.com}

\thanks{The second and fourth authors were supported in part by National Science Foundation grant 1500343.}

\date{\today}

\begin{abstract}
Subject to announced results by Geelen, Gerards, and Whittle, we completely characterize the highly connected members of the classes of dyadic, near-regular, and sixth-root-of-unity matroids.
\end{abstract}

\maketitle

\section{Introduction}
\label{sec:Introduction}
We give the definitions of the classes of dyadic, signed-graphic, near-regular, and $\sqrt[6]{1}$-matroids in Section \ref{sec:Preliminaries}; however, unexplained notation and terminology in this paper will generally follow Oxley \cite{o11}. One exception is that we denote the vector matroid of a matrix $A$ by $M(A)$ rather than $M[A]$. A matroid $M$ is \emph{vertically $k$-connected} if, for every set $X\subseteq E(M)$ with $r(X)+r(E-X)-r(M)<k$, either $X$ or $E-X$ is spanning. If $M$ is vertically $k$-connected, then its dual $M^*$ is \textit{cyclically $k$-connected}. The matroids $\Pi_r$, $\Sigma_r$, and $\Omega_r$ are obtained from $M(K_{r+1})$ by adding three specific points to a flat of rank $4$, $3$, and $5$, respectively; we give the precise definitions in Section \ref{sec:dyadic}.

Due to the technical nature of Hypotheses \ref{hyp:connected-template} and \ref{hyp:clique-template}, we delay their statements  to Section \ref{sec:Frame Templates}. Subject to these hypotheses, we characterize the highly connected dyadic matroids by proving the following.

\begin{theorem}
\label{thm:dyadic}
Suppose Hypothesis \ref{hyp:connected-template} holds. Then there exists $k\in\mathbb{Z}_+$ such that, if $M$ is a $k$-connected dyadic matroid with at least $2k$ elements, then one of the following holds\textcolor{red}{:}
\begin{enumerate}
\item Either $M$ or $M^*$ is a signed-graphic matroid.
\item Either $M$ or $M^*$ is a matroid of rank $r$ that is a restriction of $\Pi_r$, $\Sigma_r$, or $\Omega_r$.
\end{enumerate}
Moreover, suppose Hypothesis \ref{hyp:clique-template} holds. There exist $k,n\in\mathbb{Z}_+$ such that, if $M$ is a simple, vertically $k$-connected, dyadic matroid with an $M(K_n)$-minor, then either $M$ is a signed-graphic matroid or $M$ is a restriction of $\Pi_{r(M)}$, $\Sigma_{r(M)}$, or $\Omega_{r(M)}$.
\end{theorem}

Let $\mathrm{AG}(2,3)\backslash e$ be the matroid resulting from $\mathrm{AG}(2,3)$ by deleting one point; this matroid is unique up to isomorphism. We prove the following excluded-minor characterization of the highly connected dyadic matroids.

\begin{theorem}
\label{thm:dyadic-exc-minor}
Suppose Hypothesis \ref{hyp:connected-template} holds. There exists $k\in\mathbb{Z}_+$ such that, for a $k$-connected matroid $M$ with at least $2k$ elements, $M$ is dyadic if and only if $M$ is a ternary matroid with no minor isomorphic to either $\mathrm{AG}(2,3)\backslash e$ or $(\mathrm{AG}(2,3)\backslash e)^*$. Moreover, suppose Hypothesis \ref{hyp:clique-template} holds. There exist $k,n\in\mathbb{Z}_+$ such that, for a vertically $k$-connected matroid $M$ with an $M(K_n)$-minor, $M$ is dyadic if and only if $M$ is a ternary matroid with no minor isomorphic to $\mathrm{AG}(2,3)\backslash e$.
\end{theorem}

Our final main result characterizes the highly connected near-regular and $\sqrt[6]{1}$-matroids. The matroid $T_r^1$ is obtained from the complete graphic matroid $M(K_{r+2})$ by adding a point freely to a triangle, contracting that point, and simplifying. We denote the non-Fano matroid by $F_7^-$.

\begin{theorem}
\label{thm:near-reg-sixth-root}
Suppose Hypothesis \ref{hyp:connected-template} holds. There exists $k\in\mathbb{Z}_+$ such that, if $M$ is a $k$-connected matroid with at least $2k$ elements, the following are equivalent\textcolor{red}{:}
\begin{enumerate}
\item $M$ is a near-regular matroid,
\item $M$ is a $\sqrt[6]{1}$-matroid,
\item $M$ or $M^*$ is a matroid of rank $r$ that is a restriction of $T_r^1$, and
\item $M$ is a ternary matroid that has no minor isomorphic to $F_7^-$ or $(F_7^-)^*$.
\end{enumerate}
Moreover, suppose Hypothesis \ref{hyp:clique-template} holds. There exist $k,n\in\mathbb{Z}_+$ such that, if $M$ is a simple, vertically $k$-connected matroid with an $M(K_n)$-minor, then (1) and (2) are equivalent to each other and also to the following conditions\textcolor{red}{:}
\begin{enumerate}
\item[(3')] $M$ is a restriction of $T_{r(M)}^1$, and
\item[(4')] $M$ is a ternary matroid that has no minor isomorphic to $F_7^-$.
\end{enumerate}
\end{theorem}

Theorem \ref{thm:near-reg-sixth-root} leads to the following result.

\begin{corollary}
\label{cor:GF3-and-other}
Suppose Hypothesis \ref{hyp:connected-template} holds. There exists $k\in\mathbb{Z}_+$ such that, if $M$ is a ternary $k$-connected matroid with at least $2k$ elements, then $M$ is representable over some field of characteristic other than $3$ if and only if $M$ is dyadic. Moreover, suppose Hypothesis \ref{hyp:clique-template} holds. There exist $k,n\in\mathbb{Z}_+$ such that, if $M$ is a simple, ternary, vertically $k$-connected matroid with an $M(K_n)$-minor, then $M$ is representable over some field of characteristic other than $3$ if and only if $M$ is dyadic.
\end{corollary}

\begin{proof}
Whittle \cite[Theorem 5.1]{w97} showed that a $3$-connected ternary matroid that is representable over some field of characteristic other than $3$ is either a dyadic matroid or a $\sqrt[6]{1}$-matroid. Therefore, since near-regular matroids are dyadic, Theorem \ref{thm:near-reg-sixth-root} immediately implies the first statement in the corollary. The second statement is proved similarly but also requires the fact that a simple, vertically $3$-connected matroid is $3$-connected.
\end{proof}

Hypotheses \ref{hyp:connected-template} and \ref{hyp:clique-template} are believed to be true, but their proofs are still forthcoming in future papers by Geelen, Gerards, and Whittle. They are modified versions of a hypothesis given by Geelen, Gerards, and Whittle in \cite{ggw15}. The results announced in \cite{ggw15} rely on the Matroid Structure Theorem by these same authors \cite{ggw06}.  We refer the reader to \cite{gvz18} for more details.

Some proofs in this paper involved case checks aided by Version 8.9 of the SageMath software system \cite{sage}, in particular making use of the \emph{matroids} component \cite{sage-matroid}. The authors used the CoCalc (formerly SageMathCloud) online interface.

In Section \ref{sec:Preliminaries}, we give some background information about the classes of matroids studied in this paper. In Section \ref{sec:Frame Templates}, we recall results from \cite{g-submitted} that will be used to prove our main results. In Section \ref{sec:dyadic}, we prove Theorems \ref{thm:dyadic} and \ref{thm:dyadic-exc-minor}, and in Section \ref{sec:6sqrt}, we prove Theorem \ref{thm:near-reg-sixth-root}.

\section{Preliminaries}
\label{sec:Preliminaries}
We begin this section by clarifying some notation and terminology that will be used throughout the rest of the paper. Let $D_r$ be the $r\times\binom{r}{2}$ matrix such that each column is distinct and such that every column has exactly two nonzero entries---the first a $1$ and the second $-1$. For a field $\F$, we denote by $\bFp$ the prime subfield of $\F$. If $\mathcal{M}$ is a class of matroids, we will denote by $\overline{\mathcal{M}}$ the closure of $\mathcal{M}$ under the taking of minors. If $v$ is a vector in the vector space $\mathbb{F}^S$ and $S'$ is a subset of $S$, let $v|S'$ denote the projection of $v$ into $\mathbb{F}^{S'}$. The \emph{support} of $v$ is the subset $S'$ of $S$ such that all entries of $v|S'$ are nonzero and such that every entry of $v|(S-S')$ is $0$. The \emph{weight} of a column or row vector of a matrix is its number of nonzero entries; that is, the weight of a vector is the size of its support. If $A$ is an $m\times n$ matrix and $n'\leq n$, then we call an $m\times n'$ submatrix of $A$ a \emph{column submatrix} of $A$. All empty matrices (having $0$ rows or $0$ columns) are denoted by $[\emptyset]$. In the remainder of this section, we give some background information about the classes of matroids studied in this paper. 

The class of \emph{dyadic} matroids consists of those matroids representable by a matrix over $\mathbb{Q}$ such that every nonzero subdeterminant is $\pm2^i$ for some $i\in\mathbb{Z}$. The class of \emph{sixth-root-of-unity} matroids (or $\sqrt[6]{1}$-matroids) consists of those matroids that are representable by a matrix over $\mathbb{C}$ such that every nonzero subdeterminant is a complex sixth root of unity. Let $\mathbb{Q}(\alpha)$ be the field obtained by extending the rationals $\mathbb{Q}$ by a transcendental $\alpha$. A matroid is \emph{near-regular} if it can be represented by a matrix over $\mathbb{Q}(\alpha)$ such that every nonzero subdeterminant is contained in the set $\{\pm\alpha^i(\alpha-1)^j:i,j\in\mathbb{Z}\}$.

A matroid is \emph{signed-graphic} if it can be represented by a matrix over $\GF(3)$ each of whose columns has at most two nonzero entries. The rows and columns of this matrix can be indexed by the set of vertices and edges, respectively, of a \emph{signed graph}. If the nonzero entries of the column are unequal, then the corresponding edge is a positive edge joining the vertices indexing the rows containing the nonzero entries. If the column has two equal entries, the edge is negative. If the column contains only one nonzero entry, then the corresponding edge is a negative loop at the vertex indexing the row containing the nonzero entry. Every signed-graphic matroid is dyadic. (See, for example, \cite[Lemma 8A.3]{z82}).

Whittle \cite[Theorem 1.4]{w97} showed that the following statements are equivalent for a matroid $M$:
\begin{itemize}
\item $M$ is near-regular
\item $M$ is representable over $\GF(3)$, $\GF(4)$, and $\GF(5)$
\item $M$ is representable over all fields except possibly $\GF(2)$\textcolor{red}{.}
\end{itemize}
He also showed \cite[Theorem 1.2]{w97} that the class of $\sqrt[6]{1}$-matroids consists exactly of those matroids representable over $\GF(3)$ and $\GF(4)$ and \cite[Theorem 1.1]{w97} that the class of dyadic matroids consists exactly of those matroids representable over $\GF(3)$ and $\GF(5)$. Thus, the class of near-regular matroids is the intersection of the classes of $\sqrt[6]{1}$-matroids and dyadic matroids.

A geometric representation of $\mathrm{AG}(2,3)\backslash e$ is given in Figure \ref{fig:AG23-}. 
\begin{figure}[ht]
\[\begin{tikzpicture}[x=1.3cm, y=1.1cm]
	\vertex[fill,inner sep=1pt,minimum size=1pt] (1) at (0,2) [label=above:$1$] {};
 	\vertex[fill,inner sep=1pt,minimum size=1pt] (2) at (2,2) [label=above:$2$] {};
 	\vertex[fill,inner sep=1pt,minimum size=1pt] (3) at (0,0) [label=below:$3$] {};
	\vertex[fill,inner sep=1pt,minimum size=1pt] (4) at (1,2) [label=below:$4$] {};
 	\vertex[fill,inner sep=1pt,minimum size=1pt] (5) at (0,1) [label=left:$5$] {};
 	\vertex[fill,inner sep=1pt,minimum size=1pt] (6) at (2,0) [label=below:$6$] {};
 	\vertex[fill,inner sep=1pt,minimum size=1pt] (7) at (2,1) [label=right:$7$] {};
 	\vertex[fill,inner sep=1pt,minimum size=1pt] (8) at (1,0) [label=above:$8$] {};
 	\path
 		(1) edge (3)
 		(1) edge (2)
 		(3) edge (6)
 		(2) edge (6);
\draw plot [smooth] coordinates {(0,0) (3,-1) (2,1) (1,2)};
\draw plot [smooth] coordinates {(0,1) (1,2) (3,3) (2,0)};
\draw plot [smooth] coordinates {(0,2) (-1,-1) (1,0) (2,1)};
\draw plot [smooth] coordinates {(1,0) (0,1) (-1,3) (2,2)};
\end{tikzpicture}\]
\caption{A Geometric Representation of $\mathrm{AG}(2,3)\backslash e$}
  \label{fig:AG23-}
\end{figure}
It is fairly well known that $\mathrm{AG}(2,3)\backslash e$ is an excluded minor for the class of dyadic matroids. (See, for example, \cite[Section 14.7]{o11}.) We will use this fact in Section \ref{sec:dyadic}.

It is an open problem to determine the complete list of excluded minors for the dyadic matroids; however, the excluded minors for the classes of $\sqrt[6]{1}$-matroids and near-regular matroids have been determined. Geelen, Gerards, and Kapoor \cite[Corollary 1.4]{ggk00} showed that the excluded minors for the class of $\sqrt[6]{1}$-matroids are $U_{2,5}$, $U_{3,5}$, $F_7$, $F_7^*$, $F_7^-$, $(F_7^-)^*$, and $P_8$. (We refer the reader to \cite{ggk00} or \cite{o11} for the definitions of these matroids.) Hall, Mayhew, and Van Zwam \cite[Theorem 1.2]{hmvz18}, based on unpublished work by Geelen, showed that the excluded minors for the class of near-regular matroids are $U_{2,5}$, $U_{3,5}$, $F_7$, $F_7^*$, $F_7^-$, $(F_7^-)^*$, $\AG(2,3)\backslash e$, $(\AG(2,3)\backslash e)^*$, $\Delta_T(\AG(2,3)\backslash e)$, and $P_8$. Here, $\Delta_T(\AG(2,3)\backslash e)$ is the result of performing a $\Delta$-$Y$ operation on $\AG(2,3)\backslash e$.

If $r\neq3$, it follows from results of Kung \cite[Theorem 1.1]{k90} and Kung and Oxley \cite[Theorem 1.1]{ko88} that the largest simple dyadic matroid of rank $r$ is the rank-$r$ ternary Dowling geometry, which is a signed-graphic matroid. Again, suppose $r\neq3$. Then Oxley, Vertigan, and Whittle \cite[Theorem 2.1, Corollary 2.2]{ovw98} showed that $T_r^1$ is the largest simple $\sqrt[6]{1}$-matroid of rank $r$ and the largest simple near-regular matroid of rank $r$. We remark without proof that our main results here, combined with \cite[Lemmas 4.14, 4.16]{gvz18}, show that Hypothesis \ref{hyp:clique-template} agrees with these known results.

\section{Frame Templates}
\label{sec:Frame Templates}

The notion of frame templates was introduced by Geelen, Gerards, and Whittle in \cite{ggw15} to describe the structure of the highly connected members of minor-closed classes of matroids representable over a fixed finite field. Frame templates have been studied further in \cite{gvz17,nw17,gvz18,g-submitted}. In this section, we give several results proved in those papers that we will need to prove the main results in this paper. The results in \cite{ggw15} technically deal with represented matroids---which can be thought of as fixed representation matrices for matroids. However, since we only deal with ternary matroids in this paper, and since ternary matroids are uniquely $\mathrm{GF}(3)$-representable \cite{bl76}, we will state the results in terms of matroids rather than represented matroids.

A \emph{frame matrix} is a matrix each of whose columns has at most two nonzero entries. If $\F$ is a field, let $\mathbb{F}^{\times}$ denote the multiplicative group of $\mathbb{F}$, and let $\Gamma$ be a subgroup of $\mathbb{F}^{\times}$. A \emph{$\Gamma$-frame matrix} is a frame matrix $A$ such that:
\begin{itemize}
 \item Each column of $A$ with a nonzero entry contains a 1.
 \item If a column of $A$ has a second nonzero entry, then that entry is $-\gamma$ for some $\gamma\in\Gamma$.
\end{itemize}
If $\Gamma=\{1\}$, then the vector matroid of a $\Gamma$-frame matrix is a graphic matroid. For this reason, we will call the columns of a $\{1\}$-frame matrix \emph{graphic columns}.

A \textit{frame template} over a field $\mathbb{F}$ is a tuple $\Phi=(\Gamma,C,X,Y_0,Y_1,A_1,\Delta,\Lambda)$ such that the following hold\textcolor{red}{:}
\begin{itemize}
 \item [(i)] $\Gamma$ is a subgroup of $\mathbb{F}^{\times}$.
 \item [(ii)] $C$, $X$, $Y_0$ and $Y_1$ are disjoint finite sets.
 \item [(iii)] $A_1\in \mathbb{F}^{X\times (C\cup Y_0\cup Y_1)}$.
 \item [(iv)] $\Lambda$ is a subgroup of the additive group of $\mathbb{F}^X$ and is closed under scaling by elements of $\Gamma$.
 \item [(v)] $\Delta$ is a subgroup of the additive group of $\mathbb{F}^{C\cup Y_0 \cup Y_1}$ and is closed under scaling by elements of $\Gamma$.
\end{itemize}

Let $\Phi=(\Gamma,C,X,Y_0,Y_1,A_1,\Delta,\Lambda)$ be a frame template. Let $B$ and $E$ be finite sets, and let $A'\in\mathbb{F}^{B\times E}$. We say that $A'$ \textit{respects}\footnote{\label{note1}For simplicity, we will use the terms \emph{respecting} and \emph{conforming} to mean what was called \emph{virtual respecting} and \emph{virtual conforming} in \cite{gvz17} and \cite{g-submitted}. The distinction between conforming and virtual conforming is explained in \cite{gvz17}. We can do this since every matroid virtually conforming to a template is a minor of some matroid conforming to that template \cite[Lemma 3.4]{gvz17}.}  $\Phi$ if the following hold:
\begin{itemize}
 \item [(i)] $X\subseteq B$ and $C, Y_0, Y_1\subseteq E$.
 \item [(ii)] $A'[X, C\cup Y_0\cup Y_1]=A_1$.
 \item [(iii)] There exists a set $Z\subseteq E-(C\cup Y_0\cup Y_1)$ such that $A'[X,Z]=0$, each column of $A'[B-X,Z]$ is a unit vector or zero vector, and $A'[B-X, E-(C\cup Y_0\cup Y_1\cup Z)]$ is a $\Gamma$-frame matrix.
 \item [(iv)] Each column of $A'[X,E-(C\cup Y_0\cup Y_1\cup Z)]$ is contained in $\Lambda$.
 \item [(v)] Each row of $A'[B-X, C\cup Y_0\cup Y_1]$ is contained in $\Delta$.
\end{itemize}

The structure of $A'$ is shown below.

\begin{center}
\begin{tabular}{ r|c|c|ccc| }
\multicolumn{2}{c}{}&\multicolumn{1}{c}{$Z$}&\multicolumn{1}{c}{$Y_0$}&\multicolumn{1}{c}{$Y_1$}&\multicolumn{1}{c}{$C$}\\
\cline{2-6}
$X$&columns from $\Lambda$&$0$&&$A_1$&\\
\cline{2-6}
&&&&&\\
&$\Gamma$-frame matrix&unit or zero columns&\multicolumn{3}{c|}{rows from  $\Delta$}\\
&&&&&\\
\cline{2-6}
\end{tabular}
\end{center}

Now, suppose that $A'$ respects $\Phi$ and that $A\in \mathbb{F}^{B\times E}$ satisfies the following conditions:
\begin{itemize}
\item [(i)] $A[B,E-Z]=A'[B,E-Z]$
\item [(ii)] For each $i\in Z$ there exists $j\in Y_1$ such that the $i$-th column of $A$ is the sum of the $i$-th and the $j$-th columns of $A'$.
\end{itemize}
We say that such a matrix $A$ \textit{conforms}\textsuperscript{\ref{note1}} to $\Phi$.

Let $M$ be an $\mathbb{F}$-representable matroid. We say that $M$ \textit{conforms}\textsuperscript{\ref{note1}} to $\Phi$ if there is a matrix $A$ conforming to $\Phi$ such that $M$ is isomorphic to $M(A)/C\backslash Y_1$. We denote by $\mathcal{M}(\Phi)$ the set of matroids that conform to $\Phi$. If $M^*$ conforms to a template $\Phi$, we say that $M$ \textit{coconforms} to $\Phi$. We denote by $\mathcal{M}^*(\Phi)$ the set of matroids that coconform to $\Phi$.

We now state the hypotheses on which the main results are based. As stated in Section \ref{sec:Introduction}, they are modified versions of a hypothesis given by Geelen, Gerards, and Whittle in \cite{ggw15}, and their proofs are forthcoming. In their current forms, these hypotheses were stated in \cite{gvz18}.

\begin{hypothesis}[{\cite[{Hypothesis 4.3}]{gvz18}}]
\label{hyp:connected-template}
Let $\mathbb F$ be a finite field, let $m$ be a positive integer, and let 
$\mathcal M$ be a minor-closed class of $\mathbb F$-representable matroids.
Then there exist $k\in\mathbb{Z}_+$ and
frame templates $\Phi_1,\ldots,\Phi_s,\Psi_1,\ldots,\Psi_t$ such that
\begin{enumerate}
\item
$\mathcal{M}$ contains each of the classes
$\mathcal{M}(\Phi_1),\ldots,\mathcal{M}(\Phi_s)$,
\item
$\mathcal{M}$ contains the duals of the matroids in each of the classes
$\mathcal{M}(\Psi_1),\ldots,\mathcal{M}(\Psi_t)$, and
\item
if $M$ is a simple $k$-connected member of $\mathcal M$ with at least $2k$ elements
and $M$ has no $\PG(m-1,\bFp)$-minor, 
then either
$M$ is a member of at least one of the classes
$\mathcal{M}(\Phi_1),\ldots,\mathcal{M}(\Phi_s)$, or
$M^*$ is a member of at least one of the classes
$\mathcal{M}(\Psi_1),\ldots,\mathcal{M}(\Psi_t)$.
\end{enumerate}
\end{hypothesis}

\begin{hypothesis}[{\cite[{Hypothesis 4.6}]{gvz18}}]
 \label{hyp:clique-template}
Let $\mathbb F$ be a finite field, let $m$ be a positive integer, and let $\mathcal M$ be a minor-closed class of $\mathbb F$-representable matroids. Then there exist $k,n\in\mathbb{Z}_+$ and frame templates $\Phi_1,\ldots,\Phi_s,\Psi_1,\ldots,\Psi_t$ such that
\begin{enumerate}
\item $\mathcal{M}$ contains each of the classes $\mathcal{M}(\Phi_1),\ldots,\mathcal{M}(\Phi_s)$,
\item $\mathcal{M}$ contains the duals of the matroids in each of the classes $\mathcal{M}(\Psi_1),\ldots,\mathcal{M}(\Psi_t)$,
\item if $M$ is a simple vertically $k$-connected member of $\mathcal M$ with an $M(K_n)$-minor but no $\PG(m-1,\bFp)$-minor, then $M$ is a member of at least one of the classes $\mathcal{M}(\Phi_1),\ldots,\mathcal{M}(\Phi_s)$, and
\item if $M$ is a cosimple cyclically $k$-connected member of $\mathcal M$ with an $M^*(K_n)$-minor but no $\PG(m-1,\bFp)$-minor, then $M^*$ is a member of at least one of the classes $\mathcal{M}(\Psi_1),\ldots,\mathcal{M}(\Psi_t)$.
\end{enumerate}
\end{hypothesis}

If $\Phi$ and $\Phi'$ are frame templates, it is possible that $\mathcal{M}(\Phi) = \mathcal{M}(\Phi')$ even though $\Phi$ and $\Phi'$ look very different.

\begin{definition}[{\cite[{Definition 6.3}]{g-submitted}}]
\label{def:equivalent} Let $\Phi$ and $\Phi'$ be frame templates over a field $\mathbb{F}$, then the pair $\Phi,\Phi'$ are \emph{strongly equivalent} if $\mathcal{M}(\Phi)=\mathcal{M}(\Phi')$. The pair $\Phi,\Phi'$ are \emph{minor equivalent} if $\overline{\mathcal{M}(\Phi)}=\overline{\mathcal{M}(\Phi')}$.
\end{definition}

There are other notions of template equivalence (namely equivalence, algebraic equivalence, and semi-strong equivalence) given in \cite{g-submitted}, but all of these imply minor equivalence.

If $\F$ is a field and $E$ is a set, we say that two subgroups $U$ and $W$ of the additive subgroup of the vector space $\F^E$ are \emph{skew} if $U\cap W=\{0\}$. \textcolor{red}{If $U\subseteq\mathbb{F}^E$ and $\Gamma\subseteq\mathbb{F}$, let $\Gamma U=\{\gamma u|\gamma\in\Gamma, u\in U\}$.} Nelson and Walsh \cite{nw17} gave Definition \ref{def:reduced} below. 

\begin{definition}
 \label{def:reduced}
A frame template $\Phi=(\Gamma,C,X,Y_0,Y_1,A_1,\Delta,\Lambda)$ is \emph{reduced} if there is a partition $(X_0,X_1)$ of $X$ such that
\begin{itemize}
 \item $\Delta=\Gamma(\mathbb{F}^C_p\times\Delta')$ for some additive subgroup $\Delta'$ of $\mathbb{F}^{Y_0\cup Y_1}$,
\item$\mathbb{F}_p^{X_0}\subseteq\Lambda|X_0$ while $\Lambda|X_1=\{0\}$ and $A_1[X_1,C]=0$, and
\item the rows of $A_1[X_1,C\cup Y_0\cup Y_1]$ form a basis for a subspace whose additive group is skew to $\Delta$.
\end{itemize}
We will refer to the partition $X=X_0\cup X_1$ given in Definition \ref{def:reduced} as the \emph{reduction partition} of $\Phi$.
\end{definition}

The following definition and theorem are found in \cite{gvz18}.

\begin{definition}[{\cite[{Definition 5.3}]{gvz18}}]
\label{def:refined}
 
Let $\Phi=(\Gamma,C,X,Y_0,Y_1,A_1,\Delta,\Lambda)$ be a reduced frame template, with reduction partition $X=X_0\cup X_1$. If $Y_1$ spans the matroid $M(A_1[X_1,Y_0\cup Y_1])$, then $\Phi$ is \emph{refined}.
\end{definition}

\begin{theorem}[{\cite[{Theorem 5.6}]{gvz18}}]
\label{thm:Y1spanning}
If Hypothesis \ref{hyp:connected-template} holds for a class $\mathcal{M}$, then the constant $k$ and the templates $\Phi_1,\ldots,\Phi_s,\Psi_1,\ldots,\Psi_t$ can be chosen so that the templates are refined. Moreover, if Hypothesis \ref{hyp:clique-template} holds for a class $\mathcal{M}$, then the constants $k, n$, and the templates $\Phi_1,\ldots,\Phi_s,\Psi_1,\ldots,\Psi_t$ can be chosen so that the templates are refined.
\end{theorem}

A few specific templates have been given names. We list some of those now, specifically for the ternary case. Note that $\mathcal{M}(\Phi_2)$ is the class of signed-graphic matroids.

\begin{definition}[{\cite[{Definition 7.8, ternary case}]{g-submitted}}]
\label{def:minimal}
\leavevmode
\begin{itemize}
\item $\Phi_2$ is the template with all sets empty and all groups trivial except that $\Gamma=\{\pm1\}$.
\item $\Phi_C$ is the template with all groups trivial and all sets empty except that $|C|=1$ and $\Delta\cong\mathbb{Z}/3\mathbb{Z}$.
\item $\Phi_X$ is the template with all groups trivial and all sets empty except that $|X|=1$ and $\Lambda\cong\mathbb{Z}/3\mathbb{Z}$.
\item $\Phi_{Y_0}$ is the template with all groups trivial and all sets empty except that $|Y_0|=1$ and $\Delta\cong\mathbb{Z}/3\mathbb{Z}$. 
\item $\Phi_{CX}$ is the template with $Y_0=Y_1=\emptyset$, with $|C|=|X|=1$, with $\Delta\cong\Lambda\cong\mathbb{Z}/3\mathbb{Z}$, with $\Gamma$ trivial, and with $A_1=[1]$.
\item $\Phi_{CX2}$ is the template with $Y_0=Y_1=\emptyset$, with $|C|=|X|=1$, with $\Delta\cong\Lambda\cong\mathbb{Z}/3\mathbb{Z}$, with $\Gamma$ trivial, and with $A_1=[-1]$.
\end{itemize}
\end{definition}

The next lemma follows directly from \cite[{Lemma 7.9}]{g-submitted}.

\begin{lemma}
\label{lem:YCX}
The following are true: $\overline{\mathcal{M}(\Phi_{Y_0})}\subseteq\overline{\mathcal{M}(\Phi_C)}$, and $\overline{\mathcal{M}(\Phi_X)}\subseteq\overline{\mathcal{M}(\Phi_{CX})}$, and $\overline{\mathcal{M}(\Phi_X)}\subseteq\overline{\mathcal{M}(\Phi_{CX2})}$.
\end{lemma}

Frame templates where the groups $\Gamma$, $\Lambda$, and $\Delta$ are trivial are studied extensively in \cite{g-submitted}.

\begin{definition}[{\cite[{Definitions 6.9--6.10}]{g-submitted}}]
\label{def:Y-template}
A \emph{$Y$-template} is a refined frame template with all groups trivial (so $C=X_0=\emptyset$). If $A_1\in\mathbb{F}^{X\times(Y_0\cup Y_1)}$ has the form $[I_{|X|}|P_1|P_0]$, with $[I|P_1]\in\mathbb{F}^{X\times  Y_1}$, and with $P_0\in\mathbb{F}^{X\times  Y_0}$, then $\YT(P_0,P_1)$ is defined to be the $Y$-template $(\{1\},\emptyset,X,Y_0,Y_1,A_1,\{0\},\{0\})$.
\end{definition}
In all of the $Y$-templates studied in Sections \ref{sec:dyadic} and \ref{sec:6sqrt}, the matrix $P_1$ is an empty matrix. However, we use Definition \ref{def:Y-template} in order to stay consistent with \cite{g-submitted}.

The next lemma follows from \cite[Lemmas 7.16--17]{g-submitted}.

\begin{lemma}
\label{lem:Gamma-Y}
Let $\Phi$ be a frame template such that $\overline{\mathcal{M}(\Phi')}\nsubseteq\OVM$ for each template $\Phi'\in\{\Phi_X,\Phi_C,\Phi_{Y_0},\Phi_{CX},\Phi_{CX2}\}$. Then $\Phi$ is strongly equivalent to a template $(\Gamma,C,X,Y_0,Y_1,A_1,\Delta,\Lambda)$ with $C=\emptyset$ and with $\Lambda$ and $\Delta$ both trivial.
\end{lemma}

The next several results and definitions are found in \cite{g-submitted}.

\begin{lemma}[{\cite[{Theorem 7.18}]{g-submitted}}]
\label{lem:reduce-to-Y-template}
Let $\Phi$ be a refined ternary frame template. Then either $\overline{\mathcal{M}(\Phi')}\subseteq\OVM$ for some $\Phi'\in\{\Phi_X,\Phi_C,\Phi_{Y_0},\Phi_{CX},\Phi_{CX2},\Phi_2\}$, or $\Phi$ is a $Y$-template.
\end{lemma}

\begin{definition}[{\cite[{Definition 9.3}]{g-submitted}}]
\label{def:lifted}
Let $\Phi=(\Gamma,C,X,Y_0,Y_1,A_1,\Delta,\Lambda)$ be a refined frame template with reduction partition $X=X_0\cup X_1$, with $A_1[X_0,Y_1]$ a zero matrix, and with $A_1[X_1,Y_1]$ an identity matrix. Then $\Phi$ is a \emph{lifted} template.
\end{definition}

\textcolor{red}{The next result is \cite[{Lemma 9.6}]{g-submitted}, with a few additional details that can be extracted directly from the proof in \cite{g-submitted}. Note that $\Gamma$ and $C$ are the same in both templates.}

\begin{lemma}
\label{lem:lift}
\textcolor{red}{Let $\Phi=(\Gamma,C,X,Y_0,Y_1,A_1,\Delta,\Lambda)$ be a} refined frame template. \textcolor{red}{Then $\Phi$}  is minor equivalent to a lifted template \textcolor{red}{$\Phi'=(\Gamma,C,X',Y'_0,Y'_1,A'_1,\Delta',\Lambda')$. Moreover, $\Delta$ or $\Lambda$ is trivial if and only if $\Delta'$ or $\Lambda'$ is trivial, respectively.}
\end{lemma}

\begin{definition}[{\cite[{Definition 9.7}]{g-submitted}}]
 \label{def:complete}
A $Y$-template $\YT(P_0,P_1)$ is \emph{complete} if $P_0$ contains $D_{|X|}$ as a submatrix.
\end{definition}

\begin{lemma}[{\cite[{Lemma 9.9}]{g-submitted}}]
\label{lem:no-graphic}
Let $P_0$, $P_1$, $Q_1$, and $Q_2$ be matrices with the same number of rows such that no column of $P_0$ is graphic but every column of $Q_1$ and $Q_2$ is graphic. The templates $\Phi_1=\YT([P_0|Q_1],P_1)$ and $\Phi_2=\YT([P_0|Q_2],P_1)$ are minor equivalent.
\end{lemma}

\begin{definition}[{\cite[{Definition 9.12}]{g-submitted}}]
\label{def:determined-by-P0}
The $Y$-template $\YT([P|D_{|X|}],[\emptyset])$ is the complete, lifted $Y$-template \emph{determined} by $P$ and is denoted by $\YT(P)$. (This notation differs from that used in \cite{g-submitted}.)
\end{definition}

\begin{definition}[{\cite[{Definition 6.12}]{g-submitted}}]
\label{def:universal}
Let $\YT(P)$ be a complete, lifted $Y$-template. The \emph{rank-$r$ universal matroid} for $\YT(P)$ is the matroid represented by the following matrix.
\[\left[\begin{array}{c|c|c}
\multirow{2}{*}{$I_r$}&\multirow{2}{*}{$D_r$}&P\\
\cline{3-3}
&&0\\
\end{array}\right]\]
\end{definition}

It is shown in \cite[Section 9]{g-submitted} that every simple matroid conforming to $\YT(P)$ is a restriction of some universal matroid for $\YT(P)$.

We refer to \cite[Section 11.4]{o11} for the definition of generalized parallel connections of matroids.

\begin{lemma}[{\cite[{Lemma 9.13}]{g-submitted}}]
\label{lem:univ-gen-par}
The rank-$r$ universal matroid of $\YT(P)$ is the generalized parallel connection of $M(K_{r+1})$ and $M([I_m|D_m|P])$ along $M(K_{m+1})$, where $m$ is the number of rows of $P$.
\end{lemma}

Combining \cite[Lemma 9.6]{g-submitted}, \cite[Lemma 9.9]{g-submitted}, and \cite[Lemma 9.14]{g-submitted}, we obtain the following.

\begin{lemma}
\label{lem:comp-lift}
\leavevmode
\begin{itemize}
\item[(i)] Every $Y$-template is minor equivalent to the complete, lifted $Y$-template determined by a matrix the sum of whose rows is the zero vector.
\item[(ii)] Conversely, let $\Phi$ be the complete, lifted $Y$-template determined by a matrix $P$ the sum of whose rows is the zero vector. Choose any one row of $P$. Then $\Phi$ is minor equivalent to the complete, lifted $Y$-template determined by the matrix obtained from $P$ by removing that row.
\end{itemize}
\end{lemma}

The next lemma is an easy but useful result.

\begin{lemma}
\label{lem:submatrix}
Let \textcolor{red}{$P$} be a matrix with a submatrix \textcolor{red}{$P'$}. Every matroid conforming to \textcolor{red}{$\YT(P')$} is a minor of a matroid conforming to \textcolor{red}{$\YT(P)$}.
\end{lemma}

We use the next lemma to prove the excluded minor characterizations in Theorems \ref{thm:dyadic-exc-minor} and \ref{thm:near-reg-sixth-root}; it is obtained by combining Lemmas 8.1 and 8.2 of \cite{g-submitted}.

\begin{lemma}
\label{lem:connected-excluded-minors}
Let $\mathcal{M}$ be a minor-closed class of $\mathbb{F}$-representable matroids, where $\bFp$ is the prime subfield of $\F$. Let $\mathcal{E}_1$ and $\mathcal{E}_2$ be two sets of $\mathbb{F}$-representable matroids such that
\begin{itemize}
\item[(i)] no member of $\mathcal{E}_1\cup\mathcal{E}_2$ is contained in $\mathcal{M}$,
\item[(ii)] some member of $\mathcal{E}_1$ is $\bFp$-representable,
\item[(iii)] for every refined frame template $\Phi$ over $\F$ such that $\mathcal{M}(\Phi)\nsubseteq\mathcal{M}$, there is a member of $\mathcal{E}_1$ that is a minor of a matroid conforming to $\Phi$, and
\item[(iv)] for every refined frame template $\Psi$ over $\F$ such that $\mathcal{M}^*(\Psi)\nsubseteq\mathcal{M}$, there is a member of $\mathcal{E}_2$ that is a minor of a matroid coconforming to $\Psi$.
\end{itemize}
Suppose Hypothesis \ref{hyp:connected-template} holds; there exists $k\in\mathbb{Z}_+$ such that a $k$-connected $\F$-representable matroid with at least $2k$ elements is contained in $\mathcal{M}$ if and only if it contains no minor isomorphic to one of the matroids in the set $\mathcal{E}_1\cup\mathcal{E}_2$.

Moreover, suppose Hypothesis \ref{hyp:clique-template} holds; there exist $k,n\in\mathbb{Z}_+$ such that a vertically $k$-connected $\F$-representable matroid with an $M(K_n)$-minor is contained in $\mathcal{M}$ if and only if it contains no minor isomorphic to one of the matroids in $\mathcal{E}_1$.
\end{lemma}

The remaining lemmas in this section have not appeared previously, but they will be useful in Section \ref{sec:dyadic}. Recall from Definition \ref{def:lifted} that every lifted template is refined and therefore reduced. Thus, a lifted template has a reduction partition as in Definition \ref{def:reduced}.

\begin{lemma}
\label{lem:contract}
Let $\Phi=(\Gamma,C,X,Y_0,Y_1,A_1,\Delta,\Lambda)$ be a lifted template with reduction partition $X=X_0\cup X_1$. Let $P=A_1[X_1,Y_0]$, and let $S$ be any $\Gamma$-frame matrix with $|X_1|$ rows. Then $M([S|P])\in\OVM$.
\end{lemma}

\begin{proof}
Let $R=A_1[X_0,Y_0\cup C]$. Note that the following matrix conforms to $\Phi$, since $\F_p^{X_0}\subseteq\Lambda|X_0$ in a reduced template and since the zero vector is an element of $\Lambda$ and $\Delta$.

\begin{center}
\begin{tabular}{r|c|c|c|c|c|}
\multicolumn{1}{c}{}&\multicolumn{1}{c}{$H_1$}&\multicolumn{1}{c}{$H_2$}&\multicolumn{1}{c}{$Z$}&\multicolumn{1}{c}{$Y_0$}&\multicolumn{1}{c}{$C$}\\
\cline{2-6}
$X_0$&$I_{|X_0|}$&$0$&$0$&\multicolumn{2}{c|}{$R$}\\
\cline{2-6}
$X_1$&$0$&$0$&$I_{|X_1|}$&$P$&$0$\\
\cline{2-6}
&$0$&$S$&$I_{|X_1|}$&$0$&$0$\\
\cline{2-6}
\end{tabular}
\end{center}
By contracting $H_1\cup Z\cup C$ (recalling that $C$ must be contracted to obtain a matroid that conforms to $\Phi$), we obtain the desired matroid.
\end{proof}

Finally, we have a lemma about refined templates.

\begin{lemma}
\label{lem:equiv-refined}
Suppose $\Phi=(\Gamma,C,X,Y_0,Y_1,A_1,\Delta,\Lambda)$ is a refined template strongly equivalent to $\Phi'=(\Gamma',C',X',Y_0',Y_1',A_1',\Delta',\Lambda')$ with $\Delta'$ and $\Lambda'$ trivial and $C'=\emptyset$. Then $\Phi'$ is refined.
\end{lemma}

\begin{proof}
Suppose for a contradiction that $\Phi'$ is not refined. Then the set $Y_1'$ does not span the matroid $M(A_1'[X',Y_0'\cup Y_1'])$. Since $Y_0'\subseteq E(M)$ for every matroid $M$ conforming to $\Phi'$, this implies that every matroid conforming to $\Phi'$ has a cocircuit of size at most $|Y_0'|$. Therefore, since $\Phi$ and $\Phi'$ are strongly equivalent, every matroid conforming to $\Phi$ has a cocircuit of size at most $|Y_0'|$.

However, we will construct a matroid conforming to $\Phi$ containing no such cocircuit. We construct a matrix respecting $\Phi$ so that the $\Gamma$-frame matrix is $[I_k|D_k]$, where $k>|Y_0'|$, and so that every \textcolor{red}{vector chosen} from $\Delta$ and $\Lambda$ is the zero vector. In that case, contraction of $C$ has no effect on the $\Gamma$-frame matrix. Since $\Phi$ is refined, the set $Y_1$ \textcolor{red}{spans} the matroid $M(A_1[X,Y_0\cup Y_1])$. Therefore, the vector matroid of the following matrix conforms to $\Phi$, where $[I|P_1]$ has columns indexed by $Y_1$, where $P_0$ has columns indexed by $Y_0$, and where  $[I|P_1|P_0]$ has rows indexed by the elements of $X$ remaining after $C$ is contracted.
\begin{center}
\begin{tabular}{|c|c|c|c|c|c|c|c|c|c|c|}
\hline
$0$&$0$&$I$&$I$&$\cdots$&$I$&$P_1$&$P_1$&$\cdots$&$P_1$&$P_0$\\
\hline
\multirow{3}{*}{$I_k$}&\multirow{3}{*}{$D_k$}&\multirow{3}{*}{$0$}&$1\cdots1$&&&\multirow{3}{*}{$0$}&$1\cdots1$&&&\multirow{3}{*}{$0$}\\
&&&&$\ddots$&&&&$\ddots$&&\\
&&&&&$1\cdots1$&&&&$1\cdots1$&\\
\hline
\end{tabular}
\end{center}
Since $k>|Y_0'|$, every cocircuit of this matroid has size greater than $|Y_0'|$. 
\end{proof}

\section{Dyadic Matroids}
\label{sec:dyadic}

In this section, we characterize the highly connected dyadic matroids by proving Theorems \ref{thm:dyadic} and \ref{thm:dyadic-exc-minor}. First, we will need to define the families of matroids $\Pi_r$, $\Sigma_r$, and $\Omega_r$ and to prove several lemmas that build on the machinery of Section \ref{sec:Frame Templates}.

\begin{definition}
\label{def:T-matrices}
Let \begin{center}
$T_1=\begin{bmatrix}
-1&1&0\\
-1&1&0\\
1&0&1\\
1&0&1\\
\end{bmatrix}$,
$T_2=\begin{bmatrix}
-1&1&1\\
1&-1&1\\
1&1&-1\\
\end{bmatrix}$, and
$T_3=\begin{bmatrix}
-1&-1&0\\
-1&-1&0\\
1&0&-1\\
1&0&-1\\
0&1&1\\
\end{bmatrix}$.
\end{center} For $r\geq4$ (respectively $3$, $5$), we define $\Pi_r$ (respectively $\Sigma_r$, $\Omega_r$) to be the rank-$r$ universal matroid for $\YT(P)$, where $P=T_1$ (respectively $T_2$, $T_3$).
\end{definition}

Consider the rank-$3$ ternary Dowling geometry $Q_3(\GF(3)^{\times})$. This matroid contains a restriction isomorphic to $M(K_4)$, with the signed-graphic representation given in Figure \begin{figure}[ht]
\[\begin{tikzpicture}[x=1.3cm, y=1cm,
    every edge/.style={
        draw
        }
]
	\vertex[fill] (1) at (0,0) {};
	\vertex[fill] (2) at (2,0) {};
	\vertex[fill] (3) at (1,2) {};
	\path
		(1) edge[bend right=20, line width=2.0pt] (2)
		(2) edge[bend right=20] (1)
		(3) edge[bend right=20] (2)
		(2) edge[bend right=20, line width=2.0pt] (3)
		(1) edge[bend right=20] (3)
		(3) edge[bend right=20, line width=2.0pt] (1)
	;
\end{tikzpicture}\]
\caption{A signed-graphic representation of $M(K_4)$}
  \label{fig:M(K_4)}
\end{figure} \ref{fig:M(K_4)}, with negative edges printed in bold. This representation of $M(K_4)$ has been encountered before, for example in \cite{z90,g90,sq07,gvz18}. The Dowling geometry $Q_3(\GF(3)^{\times})$ can also be represented by the matrix $[I_3|D_3|T_2]$,  with the columns of $T_2$ representing the joints of the Dowling geometry.

By Lemma \ref{lem:univ-gen-par}, the matroid $\Sigma_r$ is the generalized parallel connection of the complete graphic matroid $M(K_{r+1})$ with $Q_3(\GF(3)^{\times})$ along a common restriction isomorphic to $M(K_4)$, where the restriction of $Q_3(\GF(3)^{\times})$ has the signed-graphic representation given in Figure \ref{fig:M(K_4)}. (The restriction isomorphic to $M(K_4)$ is a modular flat of $M(K_{r+1})$, which is uniquely representable over any field. Therefore, this generalized parallel connection is well-defined.)

\begin{lemma}
The matroids $\Pi_r$, $\Sigma_r$, and $\Omega_r$ are dyadic for every integer $r$ for which they are defined.
\end{lemma}

\begin{proof}
First, we show that $\Pi_4$, $\Sigma_3$, and $\Omega_5$ are dyadic. This is true for $\Sigma_3$ because it is the ternary rank-$3$ Dowling geometry. Thus, it is signed-graphic and therefore dyadic. Let $P$ be either $T_1$ or $T_3$. Consider the matrix $[I|D|P]$ but over $\GF(5)$ rather than $\GF(3)$. The authors have checked, using SageMath that the vector matroid of this matrix is isomorphic to the vector matroid of the same matrix over $\GF(3)$. Thus, $\Pi_4$ and $\Omega_5$ are representable over both $\GF(3)$ and $\GF(5)$ and are both dyadic.

Now, Lemma \ref{lem:univ-gen-par} implies that $\Pi_r$ (respectively $\Sigma_r$, $\Omega_r$) is the generalized parallel connection of $\Pi_4$, (respectively $\Sigma_3$, $\Omega_5$) and the complete graphic matroid $M(K_{r+1})$ along a common restriction isomorphic to $M(K_5)$ (respectively $M(K_4)$, $M(K_6)$). Since every complete graphic matroid is uniquely representable over every field, a result of Brylawski \cite[Theorem 6.12]{b75} implies that $\Pi_r$, $\Sigma_r$, and $\Omega_r$ are representable over both $\GF(3)$ and $\GF(5)$ and therefore dyadic. (The reader familiar with partial fields may also apply \cite[Theorem 3.1]{mwvz11} to the dyadic partial field.)
\end{proof}

We remark that, for $r\geq4$ (respectively $4$, $5$), the matroid $\Pi_r$, (respectively $\Sigma_r$, $\Omega_r$) is not signed-graphic; this can be easily checked by using SageMath to show that $\Pi_4$, (respectively $\Sigma_4$, $\Omega_5$) is not a restriction of the rank-$4$ (respectively $4$, $5$) ternary Dowling geometry.
 
The next several lemmas will be used to determine the structure of a template $\Phi$ such that $\AGE\notin\OVM$
 
\begin{lemma}
\label{lem:check-min}
If $\Phi\in\{\Phi_X,\Phi_C,\Phi_{Y_0}\,\Phi_{CX},\Phi_{CX2}\}$, then $\AGE\in\OVM$.
\end{lemma}

\begin{proof}
 Note that the following ternary matrix conforms to $\Phi_{Y_0}$.
 \[\begin{blockarray}{rrrrrrrrr}
1 & 2 & 3 & 4 & 5 & 6 & 7 & 8 & 9\\
\begin{block}{[rrrrrrrrr]}
0& 0& 0& 0& 0& 1& 1& 1& 1\\
1& 0& 0& 1& 1& 0& 0& 0& 1\\
0& 1& 0&-1& 0& 0&-1& 0& 1\\
0& 0& 1& 0&-1& 0& 0&-1& 1\\
\end{block}
\end{blockarray}\] By contracting $9$, pivoting on the first entry, we obtain the following representation of $\mathrm{AG}(2,3)\backslash e$ (with column labels matching the labels in Figure \ref{fig:AG23-}). \[\begin{blockarray}{rrrrrrrr}
1 & 2 & 3 & 4 & 5 & 6 & 7 & 8\\
\begin{block}{[rrrrrrrr]}
1&0&0&1&1&1&1&1\\
0&1&0&-1&0&1&-1&1\\
0&0&1&0&-1&1&1&-1\\
\end{block}
\end{blockarray}\] Therefore, $\AGE\in\overline{\mathcal{M}(\Phi_{Y_0})}$. Moreover, by Lemma \ref{lem:YCX}, we also have $\AGE\in\overline{\mathcal{M}(\Phi_C)}$.

Now, note that the following ternary matrix is a representation of $\mathrm{AG}(2,3)\backslash e$. Also note that the matrix conforms to $\Phi_X$, with the top row indexed by $X$ and the bottom two rows forming a $\{1\}$-frame matrix.
\begin{equation}
\label{equ:AG23-X}
\begin{blockarray}{rrrrrrrr}
1 & 2 & 3 & 4 & 5 & 6 & 7 & 8\\
\begin{block}{[rrrrrrrr]}
0&0& -1 & 0& 1& 1&-1& 1\\
1&0&  1 & 1& 1& 0& 0& 1\\
0&1&  0 &-1& 0& 1& 1&-1\\
\end{block}
\end{blockarray}
\end{equation} Therefore, $\AGE\in\overline{\mathcal{M}(\Phi_{X})}$. Moreover, by Lemma \ref{lem:YCX}, we also have $\AGE\in\overline{\mathcal{M}(\Phi_{CX})}$ and $\AGE\in\overline{\mathcal{M}(\Phi_{CX2})}$.
\end{proof}

\begin{lemma}
\label{lem:Phi2-or-Y}
\textcolor{red}{If $\Phi$} is a refined ternary frame template such that $\AGE\notin\OVM$\textcolor{red}{, t}hen either $\mathcal{M}(\Phi)\subseteq\mathcal{M}(\Phi_2)$, or $\Phi$ is a $Y$-template.
\end{lemma}

\begin{proof}
\textcolor{red}{By Lemmas \ref{lem:check-min} and \ref{lem:Gamma-Y}, $\Phi$ is strongly equivalent to a template $\Phi'=(\Gamma',C',X',Y'_0,Y'_1,A'_1,\Delta',\Lambda')$ with $C'=\emptyset$ and with $\Lambda'$ and $\Delta'$ both trivial. By Lemma \ref{lem:equiv-refined}, the template $\Phi'$ is refined. Then Lemma \ref{lem:lift} implies that $\Phi'$ is minor equivalent to some lifted template $\Phi''=(\Gamma,\emptyset,X,Y_0,Y_1,A_1,\{0\},\{0\})$.}

\textcolor{red}{If $\mathcal{M}(\Phi'')\subseteq\mathcal{M}(\Phi_2)$, then each matroid in $\mathcal{M}(\Phi)$ is a minor of a signed-graphic matroid, since $\Phi$ and $\Phi''$ are minor equivalent. Since the class of signed-graphic matroids is minor-closed, each matroid in $\mathcal{M}(\Phi)$ is signed-graphic, implying that $\mathcal{M}(\Phi)\subseteq\mathcal{M}(\Phi_2)$. If $\Phi''$ is a $Y$-template but $\mathcal{M}(\Phi'')\nsubseteq\mathcal{M}(\Phi_2)$, then $\Phi$ is a $Y$-template by minor equivalence, Lemma \ref{lem:check-min},  and Lemma \ref{lem:reduce-to-Y-template}. Thus, it suffices to prove the result for $\Phi''$.}

The only proper subgroup of the multiplicative group of $\mathrm{GF}(3)$ is the trivial group $\{1\}$. If $\Gamma$ is the trivial group $\{1\}$, then \textcolor{red}{$\Phi''$} is a $Y$-template. Thus, we may assume that $\Gamma$ is the entire multiplicative group of $\mathrm{GF}(3)$.

\textcolor{red}{Since $C=\emptyset$ and $\Phi''$ is a lifted, $A_1=[I|P]$ for some matrix $P$ with $Y_1$ indexing the columns of the identity matrix and $Y_0$ indexing the columns of $P$.} Suppose that $P$ contains a column with at least three nonzero entries. Then by column scaling and permuting of rows, we may assume that $P$ contains either $[1,1,1]^T$ or $[1,1,-1]^T$ as a submatrix. Call this submatrix $P'$, and let \textcolor{red}{$\widehat{\Phi}=(\Gamma,\emptyset,\widehat{X},\{y\},\widehat{Y_1},\widehat{A_1},\{0\},\{0\})$}, where \textcolor{red}{$\widehat{X}$} and $\{y\}$ index the rows and column of $P'$, respectively, and where \textcolor{red}{$\widehat{A_1}=[I|P']$}, with \textcolor{red}{$\widehat{Y_1}$} indexing the columns of the identity matrix. It is not difficult to see that every matroid conforming to \textcolor{red}{$\widehat{\Phi}$} is a minor of some matroid conforming to $(\Gamma,\emptyset,X,Y_0,Y_1,[I|P],\{0\},\{0\})$, which is \textcolor{red}{$\Phi''$}. Thus, $\AGE$ is not a minor of any matroid conforming to \textcolor{red}{$\widehat{\Phi}$}.

However, if $P\textcolor{red}{'}=[1,1,-1]^T$, then the representation of $\mathrm{AG}(2,3)\backslash e$ given in Matrix (\ref{equ:AG23-X}) is of the form $[S|P']$, where $S$ is a $\Gamma$-frame matrix. Thus, by Lemma \ref{lem:contract}, $\mathrm{AG}(2,3)\backslash e$ is a minor of a matroid conforming to \textcolor{red}{$\widehat{\Phi}$}. Similarly, if $P'=[1,1,1]^T$, we may take Matrix (\ref{equ:AG23-X}) and scale by $-1$ the last row and the columns indexed by $2$ and $7$. We see that $\mathrm{AG}(2,3)\backslash e$ can be represented by a matrix of the form $[S|P']$, where $S$ is a $\Gamma$-frame matrix. Again, Lemma \ref{lem:contract} implies that $\mathrm{AG}(2,3)\backslash e$ is a minor of a matroid conforming to \textcolor{red}{$\widehat{\Phi}$}.

Therefore, we deduce that every column of $P$ has at most two nonzero entries, implying that every column of every matrix conforming to \textcolor{red}{$\Phi''$} has at most two nonzero entries. Thus, \textcolor{red}{$\mathcal{M}(\Phi'')$} consists entirely of signed-graphic matroids, and \textcolor{red}{$\mathcal{M}(\Phi'')\subseteq\mathcal{M}(\Phi_2)$}.
\end{proof}

\begin{table}[!htbp]
\caption{Forbidden submatrices of $P$ where $\AGE\notin\overline{\mathcal{M}(\YT(P))}$}
\label{tab:forbidden}
\begin{center}
\begin{tabular}{|c|c||c|c|}
\hline
Matrix&Set \textcolor{red}{$S$ to Contract}&Matrix&Set \textcolor{red}{$S$ to Contract}\\
\hline
$A=\begin{bmatrix}
1\\
1\\
1\\
1\\
\end{bmatrix}$&$\{10\}$&$B=\begin{bmatrix}
1&0\\
1&0\\
1&0\\
0&1\\
0&1\\
\end{bmatrix}$&$\{15,16\}$\\
\hline
$C=\left[\begin{array}{rrr}
1&1&0\\
1&0&1\\
0&1&1\\
1&0&0\\
0&1&0\\
0&0&1\\
\end{array}\right]$&$\{0,18,23\}$&$D=\left[\begin{array}{rrr}
1&1&0\\
1&0&1\\
0&1&1\\
1&0&0\\
0&1&1\\
\end{array}\right]$&$\{0,15\}$\\
\hline
$E=\left[\begin{array}{rr}
1&0\\
1&0\\
1&1\\
0&1\\
0&-1\\
0&-1\\
\end{array}\right]$&$\{0,4,22\}$&$F=\left[\begin{array}{rr}
1&0\\
1&1\\
1&-1\\
0&1\\
0&-1\\
\end{array}\right]$&$\{0,16\}$\\
\hline
$G=\left[\begin{array}{rr}
1&0\\
1&0\\
-1&0\\
-1&1\\
0&1\\
0&-1\\
0&-1\\
\end{array}\right]$&$\{0,1,28,29\}$&$H=\left[\begin{array}{rr}
-1&1\\
-1&-1\\
1&0\\
1&0\\
0&1\\
0&-1\\
\end{array}\right]$&$\{0,15,22\}$\\
\hline
\end{tabular}
\end{center}
\end{table}

\begin{table}[!htbp]
\caption{Forbidden submatrices of $P$ where $\AGE\notin\overline{\mathcal{M}(\YT(P))}$ (Continued)}
\label{tab:forbidden2}
\begin{center}
\begin{tabular}{|c|c||c|c|}
\hline
Matrix&Set \textcolor{red}{$S$ to Contract}&Matrix&Set \textcolor{red}{$S$ to Contract}\\
\hline
\hline
$I=\left[\begin{array}{rrr}
-1&1\\
-1&-1\\
1&-1\\
1&0\\
0&1\\
\end{array}\right]$&$\{0,16\}$&$J=\left[\begin{array}{rrr}
-1&1&1\\
-1&1&0\\
1&1&1\\
1&0&1\\
\end{array}\right]$&$\{0\}$\\
\hline
$K=\left[\begin{array}{rrr}
-1&1&1\\
-1&1&0\\
1&0&1\\
1&0&1\\
0&1&0\\
\end{array}\right]$&$\{0,13\}$&$L=\left[\begin{array}{rrr}
-1&-1&1\\
-1& 1&1\\
 1&-1&1\\
 1& 1&0\\
\end{array}\right]$&$\{0\}$\\
\hline
$M=\left[\begin{array}{rrr}
-1&-1&1\\
-1& 1&1\\
 1&-1&0\\
 1& 1&1\\
\end{array}\right]$&$\{0\}$&
$N=\left[\begin{array}{rrr}
-1&-1&0\\
-1&-1&0\\
 1& 1&1\\
 1& 0&1\\
 0& 1&1\\
\end{array}\right]$&$\{0,17\}$\\
\hline
$O=\left[\begin{array}{rr}
-1&0\\
-1&-1\\
1&1\\
1&0\\
0&-1\\
0&1\\
\end{array}\right]$&$\{0,15,22\}$&&\\
\hline
\end{tabular}
\end{center}
\end{table}

\begin{lemma}
\label{lem:forbidden}
Let $\Phi$ be the complete, lifted $Y$-template determined by some matrix $P$. If $P$ contains as a submatrix some matrix listed in Table \ref{tab:forbidden} or \ref{tab:forbidden2}, then $\AGE\in\OVM$.
\end{lemma}

\begin{proof}
Let $M=M([I|D|P'])$, where $P'$ is some submatrix listed in Table \ref{tab:forbidden} or \ref{tab:forbidden2}. By Lemma \ref{lem:submatrix}, it suffices to show that $\AGE$ is a minor of $M$. For each of the matrices $A$--$O$, the authors used SageMath to show this. The computations are expedited by contracting some subset of $E(M)$ and simplifying the resulting matroid before testing if $\mathrm{AG}(2,3)\backslash e$ is a minor. If $P'$ is an $r\times c$ matrix, then $|E(M)|=\binom{r+1}{2}+c$. Label the elements of $M$, from left to right, as $\{0,1,2,\ldots,\binom{r+1}{2}+c-1\}$. If $S$ is the set to contract listed in Table \ref{tab:forbidden} or \ref{tab:forbidden2} corresponding to the matrix $P'$, then the authors used SageMath to show that the simplification of $M/S$ contains $\mathrm{AG}(2,3)\backslash e$ as a minor. The code for an example computation is given in \ref{appendix}. Using that code will ensure that the \textcolor{red}{sets} $S$ listed in Tables \ref{tab:forbidden} and \ref{tab:forbidden2} are accurate, if one wishes to reproduce these calculations.
\end{proof}

By considering the matrix $[I|D|P]$, one can see that scaling a column of one of the matrices listed in Table \ref{tab:forbidden} or \ref{tab:forbidden2} results in another forbidden submatrix. This is because scaling a column of $P$ is the same as scaling a column of $[I|D|P]$; the vector matroid of the resulting matrix is equal to that of the original matrix. Similarly, permutation of rows or columns of $P$ can be extended to permutation of rows or columns of $[I|D|P]$, resulting in a matrix with an equal vector matroid. Therefore, permutation of rows or columns of a matrix listed in Table \ref{tab:forbidden} or \ref{tab:forbidden2} results in another forbidden matrix. However, scaling a row of $[I|D|P]$ results in a matrix that can no longer be written in that form. That is, if we scale a row, the resulting matrix no longer is of the form of Definition \ref{def:universal}\textcolor{red}{.} Therefore, scaling a row of one of the matrices listed in Table \ref{tab:forbidden} or \ref{tab:forbidden2} does not necessarily result in another forbidden matrix.

\begin{notation}
In the remainder of this section, the matrices $A$--$O$ are the matrices listed in Tables \ref{tab:forbidden} and \ref{tab:forbidden2}.
\end{notation}

\begin{lemma}
\label{lem:main-case}
Let $V$ be a ternary matrix consisting entirely of unit columns, and let $W$ be a ternary matrix each of whose columns has exactly two nonzero entries, both of which are $1$s. If $\Phi$ is the complete, lifted $Y$-template determined by the following ternary matrix, then $\mathcal{M}(\Phi)$ is contained in the class of signed-graphic matroids.
\[\left[
\begin{array}{c|c}
1\cdots1&-1\cdots-1\\
1\cdots1&-1\cdots-1\\
\hline
V&W\\
\end{array}
\right]\]
\end{lemma}

\begin{proof}
Let $P$ be the matrix obtained from the given matrix by removing the top row. By Lemma \ref{lem:comp-lift}(ii), $\Phi$ is minor equivalent to the complete, lifted $Y$-template $\YT(P)$. Every matroid $M$ conforming to $\YT(P)$ is a restriction of the vector matroid of a matrix of the following form. (If the rank of $M$ is larger than the number of rows of $P$, then some of the rows of $[V|W]$ below are zero rows.)
\[\left[
\begin{array}{c|c|c|c|c|c}
1&0\cdots0&1\cdots1&0\cdots0&1\cdots1&-1\cdots-1\\
\hline
0&&&&&\\
\vdots&I&-I&D&V&W\\
0&&&&&\\
\end{array}
\right]\]
From the first row, subtract all other rows. The result is the following.
\[\left[
\begin{array}{c|c|c|c|c|c}
1&-1\cdots-1&-1\cdots-1&0\cdots0&0\cdots0&0\cdots0\\
\hline
0&&&&&\\
\vdots&I&-I&D&V&W\\
0&&&&&\\
\end{array}
\right]\]
This is a ternary matrix such that each column has at most two nonzero entries. Therefore, $M$ is a signed-graphic matroid.
\end{proof}

If the sum of the rows of a matrix is the zero vector, then the sum of the nonzero entries of each column is $0$. The next definition describes two types of columns of ternary matrices for which this is true.

\begin{definition}
\label{def:type-3-or-4}
A column of a ternary matrix is of \emph{type $3$} if it is of the form $[1,1,1,0,\ldots,0]^T$, up to permuting of rows. A column of a ternary matrix is of \emph{type $4$} if it is of the form $[-1,-1,1,1,0,\ldots,0]^T$, up to permuting of rows. 
\end{definition}

\begin{lemma}
\label{lem:3or4nonzeros}
Let $\YT(P)$ be a complete, lifted, ternary $Y$-template such that $\AGE\notin\overline{\mathcal{M}(\YT(P))}$. Then $\YT(P)$ is minor equivalent to the complete, lifted $Y$-template determined by a ternary matrix $P'$ each of whose columns can be scaled so that it is either a type-$3$ column or a type-$4$ column.
\end{lemma}

\begin{proof}
By Lemma \ref{lem:comp-lift}(i), we may assume that the sum of the rows of $P$ is the zero vector. By Lemma \ref{lem:no-graphic}, we may assume that $P$ has no graphic columns. Since graphic columns are the only columns with weight at most $2$ whose entries sum to $0$, every column of $P$ has at least three nonzero entries.

Now, by Lemma \ref{lem:comp-lift}(ii), the complete, lifted $Y$-template determined by the matrix $[1,1,1,1,-1]^T$ is minor equivalent to the complete, lifted $Y$-templates determined by both $[1,1,1,-1]^T$ and $[1,1,1,1]^T$, which is the matrix $A$ from Table \ref{tab:forbidden}. Since $A$ is forbidden from being a submatrix of $P$, so is $[1,1,1,-1]^T$ by minor equivalence. Thus, no column of $P$ can contain four nonzero entries, at least three of which are equal. Therefore, by scaling columns of $P$, we may assume that every column of $P$ is of the form $[1,1,1,0,\ldots,0]^T$ or of the form $[-1,-1,1,1,0,\ldots,0]^T$, up to permuting rows. Thus, each column of $P$ is either of type $3$ or of type $4$.
\end{proof}

\begin{lemma}
\label{lem:type-3-only}
Let $P$ be a ternary matrix all of whose columns are of type $3$. Either $\AGE\in\overline{\mathcal{M}(\YT(P))}$, or all matroids in $\mathcal{M}(\YT(P))$ are signed-graphic.
\end{lemma}

\begin{proof}
Suppose $\AGE\notin\overline{\mathcal{M}(\YT(P))}$. By Lemma \ref{lem:forbidden}, the matrices $B$, $C$, and $D$ from Table \ref{tab:forbidden} are not submatrices of $P$. Therefore, the intersection of the supports of all of the columns of $P$ is nonempty. Therefore, Lemma \ref{lem:comp-lift}(ii) implies that $\YT(P)$ is minor equivalent to the template $\Phi'$ determined by the matrix obtained from $P$ by removing the row where every column has a nonzero entry. Every matrix conforming to $\Phi'$ has at most two nonzero entries per column. Therefore, every matroid conforming to the template is signed-graphic.
\end{proof}

\begin{lemma}
\label{lem:type-3-and-4}
Let $P$ be a ternary matrix with exactly two columns, one of type $3$ and one of type $4$. Either $\AGE\in\overline{\mathcal{M}(\YT(P))}$, or the columns of $P$ can be scaled so that the resulting matrix is of one of the following forms, up to permuting of rows and columns and removal of rows all of whose entries are $0$.

\begin{center}
$\begin{bmatrix}
-1&1\\
-1&1\\
1&0\\
1&0\\
0&1\\
\end{bmatrix}$
$\begin{bmatrix}
-1&1\\
-1&1\\
1&1\\
1&0\\
\end{bmatrix}$
\end{center}
\end{lemma}

\begin{proof}
Suppose $\AGE\notin\overline{\mathcal{M}(\YT(P))}$. Then all matrices listed in Tables \ref{tab:forbidden} and \ref{tab:forbidden2} are forbidden from being submatrices of $P$. Let $E'$ be the matrix obtained by removing the row of $E$ with two nonzero entries. By Lemma \ref{lem:comp-lift}(ii), the matrix $E'$ is forbidden from being a submatrix of $P$ also.

If the columns of $P$ have supports with an intersection of size $0$ or $1$, then $P$ contains $E'$ or $E$, respectively, up to permuting rows and scaling columns. Since $E'$ and $E$ are both forbidden matrices, the two columns of $P$ must have supports with an intersection of size at least $2$. Since $F$ is a forbidden matrix, the result follows.
\end{proof}

\begin{lemma}
\label{lem:2-type-4}
Let $P$ be a ternary matrix with exactly two columns, both of type $4$ and neither a scalar multiple of the other. Either $\AGE\in\overline{\mathcal{M}(\YT(P))}$, or the columns of $P$ can be scaled so that the resulting matrix is of one of the following forms, up to permuting of rows and columns and removal of rows all of whose entries are $0$.

\begin{center}
$\begin{bmatrix}
-1&-1\\
-1&-1\\
1&0\\
1&0\\
0&1\\
0&1\\
\end{bmatrix}$
$\begin{bmatrix}
-1&-1\\
-1&-1\\
1&1\\
1&0\\
0&1\\
\end{bmatrix}$
$\begin{bmatrix}
-1&-1\\
-1&1\\
1&-1\\
1&1\\
\end{bmatrix}$
\end{center}
\end{lemma}

\begin{proof}
Suppose $\AGE\notin\overline{\mathcal{M}(\YT(P))}$. Then all matrices listed in Tables \ref{tab:forbidden} and \ref{tab:forbidden2} are forbidden from being submatrices of $P$.  Let $G'$ be the matrix obtained by removing the row of $G$ with two nonzero entries. By Lemma \ref{lem:comp-lift}(ii), the matrix $G'$ is forbidden from being a submatrix of $P$ also.

If the columns of $P$ have supports with an intersection of size $0$ or $1$, then $P$ contains $G'$ or $G$, respectively, up to  scaling of columns and permuting of rows and columns. Since $G'$ and $G$ are both forbidden matrices, the two columns of $P$ must have supports with an intersection of size at least $2$. Moreover, since the matrices $H$, $I$, and $O$  are forbidden, the result follows.
\end{proof}

Recall from Definition \ref{def:T-matrices}, that $T_1$, $T_2$, and $T_3$ are the matrices
\begin{center}
$T_1=\begin{bmatrix}
-1&1&0\\
-1&1&0\\
1&0&1\\
1&0&1\\
\end{bmatrix}$,
$T_2=\begin{bmatrix}
-1&1&1\\
1&-1&1\\
1&1&-1\\
\end{bmatrix}$, and
$T_3=\begin{bmatrix}
-1&-1&0\\
-1&-1&0\\
1&0&-1\\
1&0&-1\\
0&1&1\\
\end{bmatrix}$.
\end{center}

\begin{lemma}
\label{lem:433}
Let $P$ be a ternary matrix with at least three columns such that exactly one column is of type $4$ and all other columns are of type $3$. Also, suppose that no column of $P$ is a scalar multiple of another. Either $\AGE\in\overline{\mathcal{M}(\YT(P))}$, or all matroids in $\mathcal{M}(\YT(P))$ are signed-graphic, or $\YT(P)$ is minor equivalent to the complete, lifted $Y$-template $\YT(P')$, where $P'$ is a submatrix of $T_1$. 
\end{lemma}

\begin{proof}
Suppose $\AGE\notin\overline{\mathcal{M}(\YT(P))}$. Then all matrices listed in Tables \ref{tab:forbidden} and \ref{tab:forbidden2} are forbidden from being submatrices of $P$. There are three cases to consider.

\begin{enumerate}
\item $P$ has exactly four nonzero rows. 
\item $P$ has exactly five nonzero rows.
\item $P$ has six or more nonzero rows.
\end{enumerate}

Let $Q$ be the submatrix of $P$ consisting of its nonzero rows.

In Case 1, Lemma \ref{lem:type-3-and-4} and the fact that the matrix $J$ is not a submatrix of $P$ imply that $Q$ must be the following matrix, up to scaling of columns and permuting of rows and columns.
\[\begin{bmatrix}
-1&1&1\\
-1&1&1\\
1&1&0\\
1&0&1\\
\end{bmatrix}\]
This matrix is of the form given in Lemma \ref{lem:main-case}. Therefore, by that lemma, all matroids in $\mathcal{M}(\YT(P))$ are signed-graphic.

In Case 2, Lemma \ref{lem:type-3-and-4} implies that $Q$ must be a column submatrix of the following matrix.
 \[\begin{blockarray}{rrrrrrr}
1 & 2 & 3 & 4 & 5 & 6 & 7\\
\begin{block}{[rrrrrrr]}
      -1 & 1 & 0 & 1 & 1 & 1 & 0\\
      -1 & 1 & 0 & 1 & 1 & 0 & 1\\
       1 & 0 & 1 & 1 & 0 & 1 & 1\\
       1 & 0 & 1 & 0 & 1 & 1 & 1\\
       0 & 1 & 1 & 0 & 0 & 0 & 0\\
\end{block}
\end{blockarray}\] 
By assumption, $Q$ includes column $1$. Since $P$ has five nonzero rows, $Q$ includes column $2$ or $3$. Without loss of generality, say that $Q$ includes column $2$. Because $K$ is a forbidden matrix, $Q$ cannot contain column $6$ or $7$, and if $Q$ includes column $3$, then it cannot contain column $4$ or $5$. If $Q$ does not contain column $3$, then $Q$ (and therefore $P$) is of the form given in Lemma \ref{lem:main-case}. Therefore, by that lemma, all matroids in $\mathcal{M}(\YT(P))$ are signed-graphic. If $Q$ does contain column $3$, then Lemma \ref{lem:comp-lift}(ii) implies that $\YT(P)$ is minor equivalent to the complete lifted $Y$-template determined by the matrix $T_1$.

In Case 3, Lemma \ref{lem:type-3-and-4} implies that, up to permuting of rows and columns, $P$ contains the following submatrix, where $(a,b,c,d)$ is either $(1,1,0,0)$ or $(0,0,1,1)$.
\[\begin{bmatrix}
-1&1&a\\
-1&1&b\\
 1&0&c\\
 1&0&d\\
 0&1&0\\
 0&0&1\\
\end{bmatrix}\]
However, $(a,b,c,d)\neq(0,0,1,1)$, because $B$ is forbidden matrix. Therefore, $(a,b,c,d)=(1,1,0,0)$. In fact, we have shown that every (type-$3$) column with a nonzero entry outside of the first five rows must have its other two nonzero entries in the first two rows. It remains to consider the type-$3$ columns in $P$ all of whose nonzero entries are in the first five rows. The analysis of Case 2 shows that such a column either has two of its nonzero entries in the first two rows or is of the form $[0,0,1,1,1,0,\ldots,0]^T$. However, the presence of the second column in the matrix above, with the fact that $B$ is a forbidden matrix, imply that the column cannot be of the form $[0,0,1,1,1,0,\ldots,0]^T$. Therefore, we deduce that all type-$3$ columns have of $P$ have nonzero entries in the first two rows. Thus, $P$ is of the form given in Lemma \ref{lem:main-case}. Therefore, by that lemma, all matroids in $\mathcal{M}(\YT(P))$ are signed-graphic.
\end{proof}

\begin{lemma}
\label{lem:443}
Let $P$ be a ternary matrix with at least three columns such that exactly two columns are of type $4$ and all other columns are of type $3$. Also, suppose that no column of $P$ is a scalar multiple of another. Either $\AGE\in\overline{\mathcal{M}(\YT(P))}$, or all matroids in $\mathcal{M}(\YT(P))$ are signed-graphic. 
\end{lemma}

\begin{proof}
Suppose $\AGE\notin\overline{\mathcal{M}(\YT(P))}$. Then all matrices listed in Tables \ref{tab:forbidden} and \ref{tab:forbidden2} are forbidden from being submatrices of $P$.

First, we show that \textcolor{red}{the type-$4$ columns cannot} have equal supports. Suppose otherwise; then Lemmas \ref{lem:type-3-and-4} and \ref{lem:2-type-4} imply that $P$ contains one of the following submatrices, up to scaling columns and permuting of rows and columns.
\[\begin{bmatrix}
-1&-1&1\\
-1& 1&1\\
 1&-1&1\\
 1& 1&0\\
\end{bmatrix}
\begin{bmatrix}
-1&-1&1\\
-1& 1&1\\
 1&-1&0\\
 1& 1&1\\
\end{bmatrix}
\]
But these are the forbidden matrices $L$ and $M$.

Because the supports of the type-$4$ columns are unequal, there are two cases to consider.
\begin{enumerate}
\item The supports of \textcolor{red}{the} type-$4$ columns intersect in a set of size $3$.
\item \textcolor{red}{The supports of the type-$4$ columns intersect} in a set of size $2$.
\end{enumerate}

In Case 1, Lemma \ref{lem:2-type-4} implies that $P$ contains the following submatrix, where the third column is part of a type-$3$ column of $P$.
\[\begin{bmatrix}
-1&-1&a\\
-1&-1&b\\
 1& 1&c\\
 1& 0&d\\
 0& 1&e\\
\end{bmatrix}\]
Lemma \ref{lem:type-3-and-4}, using the first and third columns above, implies that either $a=b=1$ or $c=d=1$. Similarly, using the second and third columns above, Lemma \ref{lem:type-3-and-4} implies that either $a=b=1$ or $c=e=1$. However, since $N$ is a forbidden matrix, we must have $a=b=1$. We see then that all type-$3$ columns must have nonzero entries in the first two rows. Therefore,  $P$ is of the form given in Lemma \ref{lem:main-case}.

In Case 2, Lemma \ref{lem:2-type-4} implies that $P$ contains the following submatrix, where the third column is part of a type-$3$ column of $P$.
\[\begin{bmatrix}
-1&-1&a\\
-1&-1&b\\
 1& 0&c\\
 1& 0&d\\
 0& 1&e\\
 0& 1&f\\
\end{bmatrix}\]
If either $a=0$ or $b=0$, Lemma \ref{lem:type-3-and-4} implies that $c=d=1$ but also that $e=f=1$. This is impossible in a type-$3$ column; therefore, all type-$3$ columns have nonzero entries in the first two rows. Thus, $P$ is of the form given in Lemma \ref{lem:main-case}.

In either case, Lemma \ref{lem:main-case} implies that all matroids in $\mathcal{M}(\YT(P))$ are signed-graphic.
\end{proof}

\begin{lemma}
\label{lem:444}
Let $P$ be a ternary matrix all of whose columns are of either type $3$ or type $4$, with at least three type-$4$ columns. Also, suppose that no column of $P$ is a scalar multiple of another. Either $\AGE\in\overline{\mathcal{M}(\YT(P))}$, or all matroids in $\mathcal{M}(\YT(P))$ are signed-graphic, or $\YT(P)$ is minor equivalent to $\YT(T_2)$ or $\YT(T_3)$.
\end{lemma}

\begin{proof}
Suppose $\AGE\notin\overline{\mathcal{M}(\YT(P))}$. Then all matrices listed in Tables \ref{tab:forbidden} and \ref{tab:forbidden2} are forbidden from being submatrices of $P$. Let $T_2^+$ be the matrix obtained from $T_2$ by appending the row $[-1,-1,-1]$, and let $T_3^+$ be the matrix obtained from $T_3$ by appending the row $[0,1,1]$. We consider the following cases.
\begin{enumerate}
\item The supports of every pair of type-$4$ columns intersect in a set of size $2$.
\item There is a pair of type-$4$ columns whose supports intersect in a set of size $3$, but no pair of type-$4$ columns have equal supports.
\item There is a pair of type-$4$ columns with equal supports.
\end{enumerate}

In Case 1, Lemma \ref{lem:2-type-4} implies that $P$ contains the following submatrix, where the third column is part of a type-$4$ column of $P$.
\[\begin{bmatrix}
-1&-1&a\\
-1&-1&b\\
 1& 0&c\\
 1& 0&d\\
 0& 1&e\\
 0& 1&f\\
\end{bmatrix}\]
If $a=0$ or $b=0$, then Lemma \ref{lem:2-type-4}, using the first and third columns, implies that $c=d\neq0$ and implies that $e=f\neq0$, using the second and third columns. Since the third column is of type $4$, we have $c=-e$. By column scaling, we may assume that $e=1$. Thus, the above matrix is $T_3^+$. It is not possible for $P$ to  contain a type-$3$ column or a fourth type-$4$ column because no such column can satisfy Lemma \ref{lem:type-3-and-4} or \ref{lem:2-type-4}, respectively, with each of the three existing columns.
Therefore, $P$, when restricted to its nonzero rows, is $T_3^+$.  \textcolor{red}{Lemma \ref{lem:comp-lift}(ii) implies that $\mathcal{M}(\YT(P))$ is minor equivalent to $\YT(T_3)$.}

Thus, we may assume that $a\neq0$ and $b\neq0$. By Lemma \ref{lem:2-type-4}, we have $a=b$. Moreover, if $P$ contains a type-$3$ column, then the analysis in the proof of Lemma \ref{lem:443} (Case 2) shows that the first two entries of the column must be nonzero. Thus, in Case 1, $P$ is of the form given in Lemma \ref{lem:main-case}. By that lemma, all matroids in $\mathcal{M}(\YT(P))$ are signed-graphic.

In Case 2, Lemma \ref{lem:2-type-4} implies that $P$ contains the following submatrix, where the third column is part of a type-$4$ column of $P$.
\[\begin{bmatrix}
-1&-1&a\\
-1&-1&b\\
 1& 1&c\\
 1& 0&d\\
 0& 1&e\\
\end{bmatrix}\]
If $a=0$ or $b=0$, then Lemma \ref{lem:2-type-4} (using the first and third columns above), implies that $c=d\neq0$, but similarly, Lemma \ref{lem:2-type-4} (using the second and third columns above) implies that $c=e\neq0$. However, no type-$4$ column has three equal nonzero entries. Therefore, $a\neq0$ and $b\neq0$, and Lemma \ref{lem:2-type-4} implies that $a=b$. Thus, every type-$4$ column has equal nonzero entries in the first two rows. Moreover, if $P$ has a type-$3$ column, then the analysis in the proof of Lemma \ref{lem:443} (Case 1) implies that the first two entries of the column are nonzero. Therefore, in Case 2, $P$ is of the form given in Lemma \ref{lem:main-case}. By that lemma, all matroids in $\mathcal{M}(\YT(P))$ are signed-graphic.

In Case 3, Lemma \ref{lem:2-type-4} implies that $P$ contains the following submatrix, where the third column is part of a type-$4$ column of $P$.
\[\begin{bmatrix}
-1&-1&a\\
-1& 1&b\\
 1&-1&c\\
 1& 1&d\\
\end{bmatrix}\]
If $a=0$, then Lemma \ref{lem:2-type-4} (using the first and third columns) implies that $c=d\neq0$, but this causes the second and third columns to violate Lemma \ref{lem:2-type-4}. Therefore, $a\neq0$. Using the same argument, replacing $a$ with $b$, we see that $b\neq0$. By symmetry (which can be seen by scaling the first two columns by $-1$), we see that $c\neq0$ and $d\neq0$ also. By Lemma \ref{lem:2-type-4}, we must have $(a,b,c,d)=(-1,1,1,-1)$. Therefore, the matrix above is $T_2^+$. Moreover, $P$ can contain no other type-$4$ columns. Now, the proof of Lemma \ref{lem:443} shows that $P$ cannot contain both a type-$3$ column and a pair of type-$4$ columns with equal supports. Therefore, we see that $P$, when restricted to its nonzero rows, is $T_2^+$.  Lemma \ref{lem:comp-lift}(ii) implies that $\mathcal{M}(\YT(P))$ is minor equivalent to $\YT(T_2)$.

Therefore, the result holds in all three cases.
\end{proof}

Recall the matrices $T_1$, $T_2$, and $T_3$ given in Definition \ref{def:T-matrices}. Also recall that $\Pi_r$, $\Sigma_r$, and $\Omega_r$ are the the rank-$r$ universal matroids for $\YT(P)$, where $P=T_1$, $P=T_2$, and $P=T_3$, respectively.

\begin{lemma}
\label{lem:YT-analysis}
Let $\Phi$ be a refined ternary frame template with $\AGE\notin\OVM$. One of the following holds:
\begin{enumerate}
\item Every member of $\mathcal{M}(\Phi)$ is a signed-graphic matroid.
\item The simplification of each member of $\OVM$ is a minor of $\Pi_r$ for some $r\geq4$.
\item The simplification of each member of $\OVM$ is a minor of $\Sigma_r$ for some $r\geq3$.
\item The simplification of each member of $\OVM$ is a minor of $\Omega_r$ for some $r\geq5$.
\end{enumerate}
\end{lemma}

\begin{proof}
By Lemma \ref{lem:Phi2-or-Y}, since $\mathcal{M}(\Phi_2)$ is the class of signed-graphic matroids, we may assume that $\Phi$ is a $Y$-template. Therefore, by Lemma \ref{lem:comp-lift}(i), $\Phi$ is minor equivalent to the complete, lifted $Y$-template $\YT(P)$ determined by a matrix $P$ the sum of whose rows is the zero vector. By Lemma \ref{lem:3or4nonzeros}, we may assume that every column of $P$ is either of type $3$ or of type $4$.

Let $\widetilde{P}$ be a column submatrix of $P$ obtained by removing all but one column from each set of columns of $P$ such that every column in the set is a scalar multiple of each of the other columns in the set. Because conditions (2), (3), and (4) deal with the simplifications of members of $\OVM$ and because the class of signed-graphic matroids is closed under both parallel extensions and the taking of minors, it suffices to consider the template $\YT(\widetilde{P})$.

If $\widetilde{P}$ has exactly one column, that column is of either type $3$ or type $4$. Therefore,  Lemma \ref{lem:main-case} implies that $\mathcal{M}(\YT(\widetilde{P}))$ is contained in the class of signed-graphic matroids.

Suppose $\widetilde{P}$ has exactly two columns. If at least one of those columns is of type $3$, then Lemmas \ref{lem:type-3-only} and \ref{lem:type-3-and-4}, combined with Lemma \ref{lem:main-case}, show that $\mathcal{M}(\YT(\widetilde{P}))$ is contained in the class of signed-graphic matroids. If both columns are of type $4$, then Lemma \ref{lem:2-type-4} shows that there are three possibilities for $\widetilde{P}$. By Lemma \ref{lem:main-case}, two of those possibilities result in $\mathcal{M}(\YT(\widetilde{P}))$ being contained in the class of signed-graphic matroids. By Lemma \ref{lem:comp-lift}(ii), the third possibility results in $\YT(\widetilde{P})$ being minor equivalent to the complete, lifted $Y$-template determined by a submatrix of $T_2$.

If $\widetilde{P}$ has three or more columns, then Lemmas \ref{lem:type-3-only}, \ref{lem:433}, \ref{lem:443}, and \ref{lem:444} imply that  either $\mathcal{M}(\YT(\widetilde{P}))$ is contained in the class of signed-graphic matroids or that $\YT(\widetilde{P})$ \textcolor{red}{is} minor equivalent to the complete, lifted $Y$-template determined by a submatrix of $T_1$, $T_2$, or $T_3$.

If $\mathcal{M}(\YT(\widetilde{P}))$ is contained in the class of signed-graphic matroids, then so is $\mathcal{M}(\Phi)$. If $\YT(\widetilde{P})$ is minor equivalent to the complete, lifted $Y$-template determined by a submatrix of $T_1$, $T_2$, or $T_3$, then by definition of $\Pi_r$, $\Sigma_r$, and $\Omega_r$, the simplification of each member of $\mathcal{M}(\YT(\widetilde{P}))$ is a \textcolor{red}{minor} of $\Pi_r$ for some $r\geq4$, or $\Sigma_r$ for $r\geq3$, or $\Omega_r$ for some $r\geq5$. This implies that condition (2), (3), or (4) holds for $\Phi$.
\end{proof}

We are now ready to prove Theorem \ref{thm:dyadic}.

\begin{proof}[Proof of Theorem \ref{thm:dyadic}]
We will prove the statement in Theorem \ref{thm:dyadic} that is dependent on Hypothesis \ref{hyp:connected-template}. The statement that is dependent on Hypothesis \ref{hyp:clique-template} is proved similarly.

Recall that $\AGE$ is an excluded minor for the class of dyadic matroids. Since $\mathrm{AG}(2,3)\backslash e$ is a restriction of $\mathrm{PG}(2,3)$, Hypothesis \ref{hyp:connected-template}, with $\mathbb{F}=\mathrm{GF}(3)$ and $m=3$, implies that there exist $k\in\mathbb{Z}_+$ and frame templates $\Phi_1,\ldots,\Phi_s,\Psi_1,\ldots,\Psi_t$ such that every $k$-connected dyadic matroid $M$ with at least $2k$ elements either is contained in one of $\mathcal{M}(\Phi_1),\ldots,\mathcal{M}(\Phi_s)$ or has a dual $M^*$ contained in one of $\mathcal{M}(\Psi_1),\ldots,\mathcal{M}(\Psi_t)$. Moreover, each of $\mathcal{M}(\Phi_1),\ldots,\mathcal{M}(\Phi_s),\mathcal{M}(\Psi_1),\ldots,\mathcal{M}(\Psi_t)$ is contained in the class of dyadic matroids.

Let $M$ be a $k$-connected dyadic matroid with at least $2k$ elements. We may assume $k\geq2$; therefore, $M$ and $M^*$ are both simple. Suppose neither $M$ nor $M^*$ is signed-graphic. We must show that either $M$ or $M^*$ is a matroid of some rank $r$ that is a restriction of $\Pi_r$, $\Sigma_r$, or $\Omega_r$. We know that there is a template $\Phi$ such that either $M\in\mathcal{M}(\Phi)$, where $\Phi\in\{\Phi_1,\ldots,\Phi_s\}$ or $M^*\in\mathcal{M}(\Phi)$, where $\Phi\in\{\Psi_1,\ldots,\Psi_t\}$. In the former case, let $M'=M$; in the latter case, let $M'=M^*$. By Theorem \ref{thm:Y1spanning}, we may assume that $\Phi$ is refined. By Lemma \ref{lem:YT-analysis}, since $M'$ is not signed-graphic, $M'$ is a minor of $\Pi_r$, $\Sigma_r$, or $\Omega_r$ for some $r$.

Now, we will show that $M'$ is a restriction of $\Pi_r$, $\Sigma_r$, or $\Omega_r$ for some $r$. \textcolor{red}{Arguing inductively, it suffices to show that the following statements hold when $N'$ is the simplification of a single-element contraction of $N\in\{\Pi_r, \Sigma_r, \Omega_r\}$:
\begin{itemize}
\item If $N=\Pi_r$, then either $N'$ is signed-graphic or $N'=\Pi_{r-1}$.
\item If $N=\Sigma_r$, then either $N'$ is signed-graphic or $N'=\Sigma_{r-1}$.
\item If $N=\Omega_r$, then either $N'$ is signed-graphic, $N'=\Pi_{r-1}$, or $N'=\Omega_{r-1}$.
\end{itemize}}
Consider a universal matroid of some rank for the $Y$-template $\YT(P)$. It is the vector matroid of a matrix of the following form.
\begin{center}
\begin{tabular}{|c|c|c|c|c|c|}
\multicolumn{1}{c}{$A$}&\multicolumn{1}{c}{$B$}&\multicolumn{1}{c}{$C$}&\multicolumn{1}{c}{$D$}&\multicolumn{1}{c}{$E$}&\multicolumn{1}{c}{$F$}\\
\hline
$I$&$0$&unit columns&$0$&$D$&$P$\\
\hline
$0$&$I$&negatives of unit columns&$D$&$0$&$0$\\
\hline
\end{tabular}
\end{center}
\textcolor{red}{There are four cases to consider, depending on whether we contract an element of $A$, $B\cup C\cup D$, $E$, or $F$. Let $P\in\{T_1, T_2, T_3\}$.}

\textcolor{red}{If we contract an element of $B\cup C\cup D$ and simplify, we obtain the universal matroid of $\YT(P)$ of rank $r-1$. Therefore, $N$ is $\Pi_r$, $\Sigma_r$, or $\Omega_r$ and $N'$ is $\Pi_{r-1}$, $\Sigma_{r-1}$, or $\Omega_{r-1}$, respectively.}

\textcolor{red}{If we contract an element of $A$ and simplify, then the result is the universal matroid of a complete, lifted $Y$-template determined by a matrix $P'$ obtained from $P$ by deleting one row and possibly removing any graphic or duplicate columns that may result. If $P'$ is obtained from $T_3$ by deleting the last row, then $P' = T_1$, up to column scaling. This implies that $N=\Omega_r$ and $N'=\Pi_{r-1}$. In all other cases, one can use Lemma \ref{lem:comp-lift}(ii) to show that $\YT(P')$ is minor equivalent to $\YT(P'')$ for some matrix $P''$ of the form given in Lemma \ref{lem:main-case}, implying that $N'$ is signed-graphic.}

\textcolor{red}{If we contract an element of $E$ and simplify, we obtain a universal matroid of a complete, lifted $Y$-template determined by a matrix $P' $ obtained by adding one row of $P$ to another, deleting the added row, and possibly removing graphic or duplicate columns that may result. One can use Lemma \ref{lem:comp-lift}(ii) to check that $\YT(P')$ is minor equivalent to $\YT(P'')$ for some matrix $P''$ such that either $P''$ is of the form given in Lemma \ref{lem:main-case} (and therefore $N'$ is signed-graphic) or $P'' = T_1$. We have $P''=T_1$ only when $P=T_3$, so in this case, $N=\Omega_r$ and $N'=\Pi_{r-1}$.}

\textcolor{red}{If $P=T_1$ and we contract the element of $F$ represented by either the second or third column of $T_1$, then we obtain a matroid represented by a matrix with at most two nonzero entries per column. In all other cases, if we contract an element of $F$ and simplify, then we obtain the universal matroid of a complete, lifted $Y$-template determined by some matrix $P'$. Using Lemma \ref{lem:comp-lift}(ii), one can check that $\YT(P')$ is minor equivalent to $\YT(P'')$ for some matrix $P''$ of the form given in Lemma \ref{lem:main-case}. Therefore, $N'$ is signed-graphic.}
\end{proof}

Now we prove Theorem \ref{thm:dyadic-exc-minor}.

\begin{proof}[Proof of Theorem \ref{thm:dyadic-exc-minor}]
We will use Lemma \ref{lem:connected-excluded-minors} to prove Theorem \ref{thm:dyadic-exc-minor}. Let $\mathcal{E}_1=\{\AGE\}$, and let $\mathcal{E}_2=\{(\AGE)^*\}$. Since $\AGE$ is ternary but not dyadic, and since the class of dyadic matroids is closed under duality, conditions (i) and (ii) of Lemma \ref{lem:connected-excluded-minors} are satisfied. Since every signed-graphic matroid is dyadic and since $\Pi_r$, $\Sigma_r$, and $\Omega_r$ are dyadic for every $r$, Lemma \ref{lem:YT-analysis} implies that condition (iii) of Lemma \ref{lem:connected-excluded-minors} is satisfied. Then (iv) follows from (iii) and duality. The result follows.
\end{proof}

\section{Near-Regular and $\sqrt[6]{1}$-Matroids}
\label{sec:6sqrt}

In this section, we prove Theorem \ref{thm:near-reg-sixth-root} after proving several lemmas.

\begin{lemma}
\label{lem:check-min6}
Let $\Phi\in\{\Phi_2,\Phi_X,\Phi_C,\Phi_{Y_0}\,\Phi_{CX},\Phi_{CX2}\}$. Then the non-Fano matroid $F_7^-$ is a minor of some member of $\mathcal{M}(\Phi)$. Therefore,  $\mathcal{M}(\Phi)$ is not contained in the class of $\sqrt[6]{1}$-matroids.
\end{lemma}

\begin{proof}
It is well known (see \cite[Proposition 6.4.8]{o11}) that $F_7^-$ is $\F$-representable if and only if the characteristic of $\F$ is not $2$. Thus, $F_7^-$ is not representable over $\GF(4)$ and is therefore not a $\sqrt[6]{1}$-matroid. Therefore, the last statement of the lemma follows from the first part of the lemma.

We saw in Section \ref{sec:dyadic} that $M(K_4)$ can be obtained from the rank-$3$ ternary Dowling geometry by deleting the three joints. If we leave one of the joints in place, then it is contained in the closures of exactly two pairs of points that are not contained in a nontrivial line of $M(K_4)$. The result is the non-Fano matroid $F_7^-$. Therefore, $F_7^-$ is signed-graphic, implying that $F_7^-\in\mathcal{M}(\Phi_2)$.

The following ternary matrix is a representation of $F_7^-$ that conforms to both $\Phi_X$ and $\Phi_{Y_0}$.
\begin{equation}
\label{equ:F7-}
\left[\begin{array}{rrrrrrr}
 1&0&0&-1&-1&0 &1\\
0&1&0&1  &  0&1 &1\\
0&0&1&0  &  1&-1&-1\\
\end{array}\right]
\end{equation}
Therefore, $F_7^-\in\mathcal{M}(\Phi_X)$ and $F_7^-\in\mathcal{M}(\Phi_{Y_0})$. By Lemma \ref{lem:YCX}, we also have that $F_7^-$ is a minor of some member of $\mathcal{M}(\Phi_C)$, some member of $\mathcal{M}(\Phi_{CX})$, and some member of $\mathcal{M}(\Phi_{CX2})$.
\end{proof}

Recall that the matroid $T_r^1$ is obtained from the complete graphic matroid $M(K_{r+2})$ by adding a point freely to a triangle, contracting that point, and simplifying.  For $r\geq2$, Semple (see \cite[Section 2]{s99} and \cite[Proposition 3.1]{s96}) showed that $T_r^1$ is representable over a field $\F$ if and only if $\F\neq\GF(2)$. Therefore, $T_r^1$ is near-regular for every $r\geq2$. The following matrix represents $T_r^1$ over every field of characteristic other than $2$.

\begin{equation}
\label{equ:Tr1}
\left[
\begin{array}{c|c|c}
\multirow{2}{*}{$I_r$}&\multirow{2}{*}{$D_r$}&1\cdots1\\
\cline{3-3}
&&I_{r-1}\\
\end{array}
\right]
\end{equation}

\begin{lemma}
\label{lem:Y-template-nonFano}
Let $\YT(P)$ be the complete, lifted $Y$-template determined by some ternary matrix $P$. Either $F_7^-$ is a minor of some matroid conforming to $\YT(P)$, or the simplification of each member of $\mathcal{M}(\YT(P))$ is a restriction of $T_r^1$ for some $r\geq2$.
\end{lemma}

\begin{proof}
Suppose that $F_7^-$ is not a minor of any matroid conforming to $\YT(P)$. We will show that the simplification of each member of $\mathcal{M}(\YT(P))$ is a restriction of $T_r^1$ for some $r\geq2$.

By Lemma \ref{lem:no-graphic}, we may assume that $P$ has no graphic columns. Suppose $P=[1,1,1,1]^T$. Consider the vector matroid of $[I_4|D_4|P]$. By contracting the element represented by $P$, we obtain a matroid containing $F_7^-$ as a restriction. Therefore, no column of $P$ can contain four equal nonzero entries. If $P=[1,1,-1]^T$, then $[I_3|D_3|P]$ is (up to column scaling) the representation of $F_7^-$ given in Matrix (\ref{equ:F7-}). Therefore, if a column of $P$ contains unequal nonzero entries, then it can contain no other nonzero entry. Thus, the column is a graphic column, which contradicts our assumption above.

Thus, every column of $P$ contains at most three nonzero entries, all of which are equal. Consider the following matrices.
\begin{center}
$\begin{bmatrix}
1&0\\
1&0\\
0&1\\
0&1\\
\end{bmatrix}$,
$\begin{bmatrix}
1&1&0\\
1&0&1\\
0&1&1\\
\end{bmatrix}$
\end{center}
If $P$ is either of these matrices, then $M([I|D|P])$ contains $F_7^-$ as a minor. (This can be easily checked with SageMath or by hand. If $P$ is the first matrix, contracting one of the elements represented by a column of $P$ and simplifying results in the rank-$3$ ternary Dowling geometry. We saw in the proof of Lemma  \ref{lem:check-min6} that $F_7^-$ is signed-graphic, implying that it is a restriction of the rank-$3$ ternary Dowling geometry. If $P$ is the second matrix, then $M([I|D|P])$ is itself the rank-$3$ ternary Dowling geometry.) Therefore, $P$ does not contain either of these matrices as a submatrix.

It is routine to check that the class of matroids whose simplifications are restrictions of some $T_r^1$ is minor-closed. Therefore, it suffices to consider a template that is minor equivalent to $\YT(P)$. Thus, by Lemma \ref{lem:comp-lift}(i), we may assume that the sum of the rows of $P$ is the zero vector. Therefore, we may assume that every column has exactly three nonzero entries all of which are equal. Because the two matrices above are forbidden, $P$ is of the following form, where $V$ consists entirely of unit columns.
\begin{center}
$\begin{bmatrix}
1&\cdots&1\\
1&\cdots&1\\
\hline
&V&\\
\end{bmatrix}$
\end{center}
Now, Lemma \ref{lem:comp-lift}(ii) implies that $\YT(P)$ is minor equivalent to the complete, lifted $Y$-template $\YT(P')$ determined by the matrix $P'$ obtained by removing the top row from $P$. The simplification of each member of $\mathcal{M}(\YT(P'))$ is a restriction of $T_r^1$ for some $r\geq2$.
\end{proof}

We are now ready to prove Theorem \ref{thm:near-reg-sixth-root}.

\begin{proof}[Proof of Theorem \ref{thm:near-reg-sixth-root}]
We prove the statement in Theorem \ref{thm:near-reg-sixth-root} that is dependent on Hypothesis \ref{hyp:connected-template}. The statement dependent on Hypothesis \ref{hyp:clique-template} is proved similarly.

Every near-regular matroid is also a $\sqrt[6]{1}$-matroid; therefore, (1) implies (2). We will now show that (2) implies (3). By Hypothesis \ref{hyp:connected-template}, there exist ternary frame templates $\Phi_1,\ldots,\Phi_s,\Psi_1,\ldots,\Psi_t$ and a positive integer $k_1$ such that every matroid conforming to these templates is a $\sqrt[6]{1}$-matroid and such that every simple $k_1$-connected $\sqrt[6]{1}$-matroid with at least $2k_1$ elements either conforms or coconforms to one of these templates. By Theorem \ref{thm:Y1spanning}, we may assume that these templates are refined. Lemmas \ref{lem:check-min6} and \ref{lem:reduce-to-Y-template} imply that each of these templates is a $Y$-template. Then Lemmas \ref{lem:comp-lift} and \ref{lem:Y-template-nonFano} imply that the simplification of every matroid conforming to these templates is isomorphic to a restriction of $T_r^1$ for some $r\geq2$ (because the class of matroids whose simplifications are restrictions of $T_r^1$ is minor-closed). We see that we may choose $r=r(M)$ because otherwise there are rows of the matrix representing $M$ that can be removed without changing the matroid. By taking $k\geq k_1$, we see that (2) implies (3).

Since $T_r^1$ is near-regular for every $r$, (3) implies (1). We complete the proof of the theorem by showing the equivalence of (2) and (4) using Lemma \ref{lem:connected-excluded-minors}. In that lemma, let $\mathcal{M}$ be the class of $\sqrt[6]{1}$-matroids, let $\mathcal{E}_1=\{F_7^-\}$, and let $\mathcal{E}_2=\{(F_7^-)^*\}$. Since $F_7^-$ and $(F_7^-)^*$ are ternary matroids that are not $\sqrt[6]{1}$-matroids, conditions (i) and (ii) of Lemma \ref{lem:connected-excluded-minors} are satisfied. Combining Lemmas \ref{lem:check-min6} and \ref{lem:reduce-to-Y-template}, we see that for every refined frame template $\Phi$ that is not a $Y$-template, $F_7^-\in\OVM$. Combining Lemmas \ref{lem:comp-lift}(i), and \ref{lem:Y-template-nonFano}, we see that $F_7^-\in\OVM$ for every $Y$-template $\Phi$ such that $\OVM$ is not contained in the class of matroids whose simplifications are restrictions of some $T_r^1$. Since $T_r^1$ is a $\sqrt[6]{1}$-matroid for every $r$, condition (iii) of Lemma \ref{lem:connected-excluded-minors} holds. Finally, condition (iv) of Lemma \ref{lem:connected-excluded-minors} holds because of condition (iii) and the fact that the class of $\sqrt[6]{1}$-matroids is closed under duality. Therefore, by Lemma \ref{lem:connected-excluded-minors}, there is a positive integer $k_2$ such that a $k_2$-connected ternary matroid with at least $2k_2$ elements is a $\sqrt[6]{1}$-matroid if and only if it contains no minor isomorphic to either $F_7^-$ or $(F_7^-)^*$. By taking $k\geq k_2$, we see that (2) is equivalent to (4), completing the proof.
\end{proof}

\renewcommand{\thesection}{Appendix A}
\section{SageMath Code}
\label{appendix}
\setcounter{theorem}{0}

We give here an example of the code for the computations used to prove Lemma \ref{lem:forbidden}. The function \texttt{complete\_Y\_template\_matrix} takes as input a matrix $P$ and returns the matrix $[I|D|P]$. The code below returns \texttt{True}, showing that if $P$ is the matrix $C$ from Table \ref{tab:forbidden}, then $M([I|D|P])/\{0,18,23\}$ contains $\mathrm{AG}(2,3)\backslash e$ as a minor.

\begin{verbatim}
N=matroids.named_matroids.AG23minus()
# N is the matroid AG(2,3)\e

def complete_Y_template_matrix(P):
    k=P.nrows()
    num_elts=k+k*(k-1)/2+P.ncols()
    F=P.base_ring()
    A = Matrix(F, k, num_elts)
    # identity in front
    for j in range(k):
        A[j,j] = 1
    i = k
    # all pairs
    for S in Subsets(range(k),2):
        A[S[0],i]=1
        A[S[1],i]=-1
        i = i + 1
    # Columns from P
    for l in range(P.ncols()):
        for j in range(k):
            A[j, i] = P[j, l]
        i=i+1
    return A
    
P =  Matrix(GF(3), [[1,1,0],
                    [1,0,1],
                    [0,1,1],
                    [1,0,0],
                    [0,1,0],
                    [0,0,1]])
A=complete_Y_template_matrix(P)
M=Matroid(field=GF(3), matrix=A)
((M/0/18/23).simplify()).has_minor(N)
\end{verbatim}

\section*{Acknowledgements}
We thank the anonymous referees for carefully reading the manuscript and giving many suggestions to improve it. In particular, they pointed out a significant error in a previous version.

\end{document}